\documentclass[smallextended]{svjour3}
\usepackage[utf8]{inputenc}
\usepackage{cite}
\usepackage{amsmath}
\usepackage{amsfonts}
\usepackage{amssymb}
\usepackage{subfigure}
\usepackage{enumerate}
\usepackage{hyperref}
\usepackage[colorinlistoftodos]{todonotes}
\usepackage{fullpage}

\newcommand{\al}{\alpha}
\newcommand{\be}{\beta}
\newcommand{\ga}{\gamma}

\newcommand{\de}{\delta}
\newcommand{\eps}{\varepsilon}
\newcommand{\bx}{\bar x}

\newcommand{\iv}{^{-1} }

\newcommand {\R} {\mathbb R}

\newcommand {\B} {\mathbb B}

\newcommand {\dom} {{\rm dom}\,}

\newcommand {\bd} {{\rm bd}\,}
\newcommand {\cone} {{\rm cone}\,}

\newcommand {\Limsup} {\mathop{{\rm Lim\,sup}\,}}
\newcommand {\sd} {\partial}
\newcommand {\Int} {{\rm int}\,}

\def\RHS{right-hand side}
\def\SVM{set-valued mapping}

\newcommand{\norm}[1]{\left\Vert#1\right\Vert}

\newcommand{\set}[1]{\left\{#1\right\}}

\newcommand{\ang}[1]{\left\langle #1 \right\rangle}
\newcommand{\qdtx}[1]{\quad\mbox{#1}\quad}
\newcounter{mycount}

\newcommand{\tr}{{\rm tr}[A,B](\bx)}
\newcommand{\str}{{\rm str}[A,B](\bx)}
\newcommand{\itr}{{\rm itr}[A,B](\bx)}
\newcommand\xqed{%
  \leavevmode\unskip\penalty9999 \hbox{}\nobreak\hfill
  \quad\hbox{$\triangle$}}


\newcommand{\strc}{{\rm str}_c[A,B](\bx)}
\newcommand{\itrd}[1]{{\rm itr}_{#1}[A,B](\bx)}
\newcommand{\itrdd}[1]{{\rm itr}_{#1}'[A,B](\bx)}
\newcommand{\itrdh}[1]{\widehat{\rm itr}_{#1}[A,B](\bx)}
\smartqed

\renewcommand{\labelenumi}{\rm (\arabic{enumi})}

\renewcommand\labelenumi{\rm(\theenumi)}

\title{About intrinsic transversality of pairs of sets
\thanks{The research was supported by Australian Research Council, project DP160100854.
}}
\author{
Alexander Y. Kruger
}
\institute{Alexander Y. Kruger \at
Centre for Informatics and Applied Optimization, Federation University Australia, POB 663, Ballarat, VIC 3350, Australia \\
\email{a.kruger@federation.edu.au}
}
\dedication{Dedicated to the memory of Professor Jonathan Michael Borwein}
\date{Received: date / Accepted: date}
\begin{document}

\maketitle

\begin{abstract}
The article continues the study of the `regular' arrangement of a collection of sets near a point in their intersection.
Such regular intersection or, in other words, \emph{transversality} properties are crucial for the validity of qualification conditions in optimization as well as subdifferential, normal cone and coderivative calculus, and convergence analysis of computational algorithms.
One of the main motivations for the development of the transversality theory of collections of sets comes from the convergence analysis of alternating projections for solving feasibility problems.
This article targets infinite dimensional extensions of the \emph{intrinsic transversality} property introduced recently by Drusvyatskiy, Ioffe and Lewis as a sufficient condition for local linear convergence of alternating projections.
Several characterizations of this property are established involving new limiting objects defined for pairs of sets.
Special attention is given to the convex case.


\end{abstract}
\section{Introduction}

This article continues the study of the `regular' arrangement of a collection of sets near a point in their intersection.
Such \emph{regular intersection} or, in other words, \emph{transversality} properties are crucial for the validity of \emph{qualification conditions} in optimization as well as subdifferential, normal cone and coderivative calculus, and convergence analysis of computational algorithms.
This explains the growing interest of researchers to investigating this class of properties and obtaining primal and dual necessary and/or sufficient conditions in various settings (convex or nonconvex, finite or infinite dimensional, finite or infinite collections); cf. Bauschke and Borwein \cite{BauBor93,BauBor96}, Ngai and Th\'era \cite{NgaThe01}, Ng and Yang \cite{NgYan04}, Bakan et al. \cite{BakDeuLi05}, Kruger et al. \cite{Kru05,Kru06,Kru09,KruLop12.1, KruTha13,KruTha14,KruTha15,KruTha16, KruLukTha,KruLukTha2}, Chong Li et al. \cite{LiNg05,LiNgPon07,LiNg14}, Ng and Zhang \cite{NgZan07}, Lewis et al. \cite{LewMal08,LewLukMal09}, Zheng et al. \cite{ZheNg08,ZheWeiYao10}, Bauschke et all \cite{BauLukPhaWan13.1,BauLukPhaWan13.2}, Hesse and Luke \cite{HesLuk13}, Drusvyatskiy et al. \cite{DruIofLew15},  Noll and Rondepierre \cite{NolRon16}.
Note also the very well known connections (in fact, equivalences) between transversality properties of collections of sets and the corresponding regularity properties of \SVM s.
For example, the properties of transversality and subtransversality of pairs of sets correspond in a sense to metric regularity and metric subregularity of set-valued mappings, respectively (cf. \cite{Iof00_,Iof16, Kru05,Kru06,Kru09,KruTha15,KruLukTha,KruLukTha2}).

Due to the wide variety of applications coming from different areas, some transversality properties together with the corresponding necessary and/or sufficient conditions have been rediscovered many times in different contexts and often under different names.
The \emph{intrinsic transversality} property studied in the current article was originally introduced in 2015 by Drusvyatskiy et al. \cite{DruIofLew15} as
an important
sufficient condition for local linear convergence of alternating projections for solving finite dimensional nonconvex feasibility problems.
The new term has not been immediately accepted: in \cite{KruTha16} the property is referred to as \emph{DIL-restricted regularity} by the first letters of the names of the three authors.
Another (unnamed) transversality property has appeared in \cite[Theorem~4(ii)]{KruLukTha}, also in the finite dimensional setting, and has been used alongside intrinsic transversality (see \cite[Theorem~4(iii)]{KruLukTha}) as a dual space sufficient condition for a much better known property called \emph{subtransversality}.
A more general and refined infinite dimensional version of the property from \cite[Theorem~4(ii)]{KruLukTha} has been formulated in \cite{KruLukTha2} and proved to imply subtransversality in Asplund spaces.
Its thorough analysis is continued in the current article with several new limiting and other characterizations produced, and special attention given to the convex case.
It has come as a surprise that, when reduced to finite dimensional Euclidean spaces, this property is equivalent (see Theorem~\ref{T6} below) to intrinsic transversality as defined by Drusvyatskiy et al. \cite{DruIofLew15}.
Although the definition is different from the one in \cite{DruIofLew15}, here and in \cite{KruLukTha2} the name `intrinsic transversality' is adopted for this property in both finite and infinite dimensions.

The origins of the concept of regular arrangement of sets in space can be traced back to that of \emph{transversality} in differential geometry (see, for instance, \cite{GuiPol74,Hir76}).
Given smooth manifolds $A$ and $B$ in a finite dimensional normed linear space with a point $\bx\in A\cap B$, their transversality
can be characterized in dual terms:
\begin{align}\label{1}
N_{A}(\bar{x}) &\cap N_{B}(\bar{x})= \{0\},
\end{align}
where $N_{A}(\bar{x})$ and $N_{B}(\bar{x})$ are the \emph{normal spaces} (i.e., orthogonal complements to the tangent spaces) to $A$ and $B$, respectively, at the point $\bx$.

Since the pioneering work by Bauschke and Borwein \cite{BauBor93} in 1993, a strong motivation for the development of the transversality theory of collections of sets has been coming from the convergence analysis of alternating (or cyclic) projections for solving feasibility problems.
Given two sets $A$ and $B$, the \emph{feasibility problem} consists in finding a point in their intersection $A\cap B$.
This is a very general model which includes, in particular, solving systems of all sorts of equations and inequalities (algebraic, differential, etc.).

Assuming for simplicity that $A$ and $B$ are closed sets in finite dimensions, \emph{alternating projections} are determined by a sequence $(x_k)$ alternating between the sets:
\begin{equation*}
x_{2k+1}\in P_B(x_{2k}),\quad x_{2k+2}\in P_A(x_{2k+1})\quad(k=0,1,\ldots),
\end{equation*}
with some initial point $x_0$; see Fig.~\ref{F1}.
\begin{figure}[htbp]
\centering
\subfigure[Linear convergence]{
\includegraphics[height=5cm]{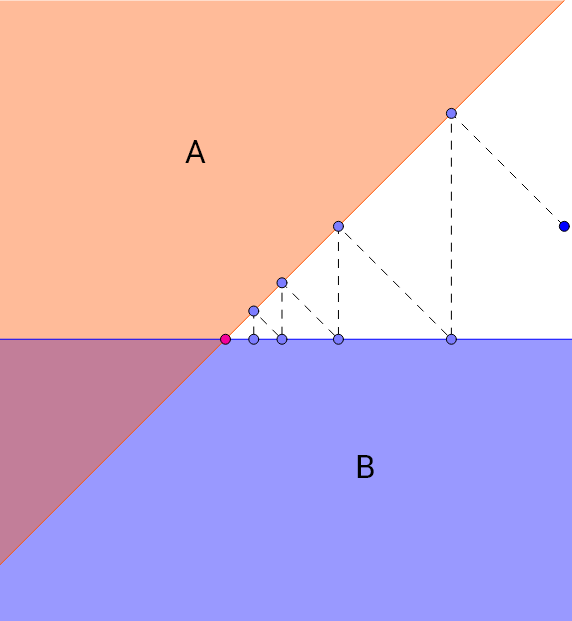}}
\qquad
\subfigure[No linear convergence]{
\includegraphics[height=5cm]{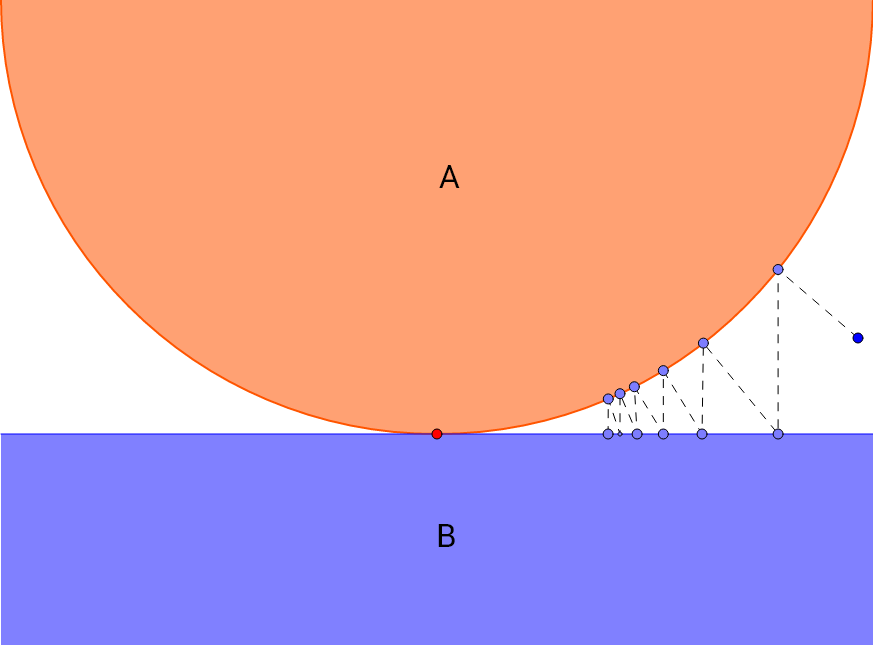}}
\caption{Alternating projections}\label{F1}
\end{figure}
Here $P_A$ and $P_B$ stand for the projection operators (see \eqref{Pr}) on the respective sets, corresponding to the Euclidean norm.
Equivalently, one can talk about a sequence $(x_k)$ defined using the composition of projection operators:
\begin{equation*}
x_{k+1}\in P_AP_B(x_{k})\quad(k=0,1,\ldots).
\end{equation*}
This simple algorithm has a long history.
It is often referred to as \emph{von Neumann method}, although some traces of this method can be found in the 19th century's publications (see the comments in \cite{NolRon16}).

Up until very recently, the method of alternating projections has been mostly studied in the convex setting.
If the sets $A$ and $B$ are convex, the projections are unique, and if $A\cap B\ne\emptyset$, the sequence always converges to a point in $A\cap B$; see Bregman \cite{Bre65} and Gurin et al \cite{GurPolRai67},  Bauschke and Borwein \cite{BauBor93}.
However, as one can see from comparing the two illustration in Fig.~\ref{F1}, the type of convergence can be strongly different.
Fig.~\ref{F1}(a) represents the case of \emph{linear convergence} characterized by the inequalities
\begin{gather*}
\norm{x_k-\hat x}\le\al c^k\quad(k=1,2,\ldots),
\end{gather*}
where $\hat x\in A\cap B$ is the limit of the sequence, $\al>0$ and $c\in]0,1[$ is the \emph{rate of convergence}.
In the case represented in Fig.~\ref{F1}(b), the above linear estimates do not hold, and the convergence obviously slows down.
It is easy to realize that the type of convergence and its rate are determined by the way the sets intersect.
For the linear convergence of alternating projections, the sets must intersect in a certain regular way.

A systematic analysis of the convergence of alternating projections in the convex setting was done by Bauschke and Borwein \cite{BauBor93,BauBor96}.
In particular, they demonstrated (see \cite[Corollary~3.14]{BauBor93}) that alternating projections converge linearly with rate $\sqrt{1-\al^2}$, provided that the pair $\{A,B\}$ of sets with $A\cap B\ne\emptyset$ is \emph{linearly regular} with rate $\al\in]0,1[$:
\begin{gather}\label{LR}
\alpha d\left(x,A\cap B\right)\le \max\left\{d(x,A),d(x,B)\right\}
\quad\mbox{for all}\quad x.
\end{gather}
Clearly, this is the case in the example in Fig.~\ref{F1}(a), while the pair of convex sets in Fig.~\ref{F1}(b) is not linearly regular.
It has been shown very recently by Luke et al. \cite{LukThaTeb} that linear regularity of the pair of convex sets with nonempty intersection is not only sufficient for the linear convergence of alternating projections, but is
also necessary.
This last result together with the theory developed by Bauschke and Borwein in the 1990s make the picture in the convex setting complete and positions the linear regularity property \eqref{LR} as the core regularity property for a pair of convex sets with nonempty intersection.

The picture becomes much more complicated
if the convexity assumption is dropped.
First, one can obviously talk only about local convergence and local (near a point in the intersection) regularity/transversality properties.
The local version of the linear regularity property \eqref{LR} --- called in this article \emph{subtransversality} (see Definition~\ref{D1}(i)) --- remains a necessary condition for certain types of local linear convergence of alternating projections; cf. \cite{LukThaTeb}.
This property has been thoroughly studied in \cite{KruLukTha2}.
On the other hand, a simple example in Fig.~\ref{F2}(b) shows that it is not sufficient to guarantee (any) convergence of alternating projections.
\begin{figure}[htbp]
\centering
\subfigure[Convex case]{
\includegraphics[height=5cm]{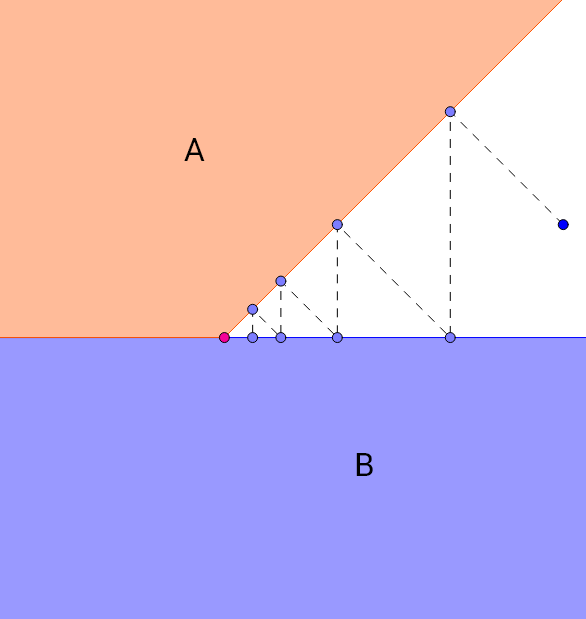}}
\qquad
\subfigure[Nonconvex case]{
\includegraphics[height=5cm]{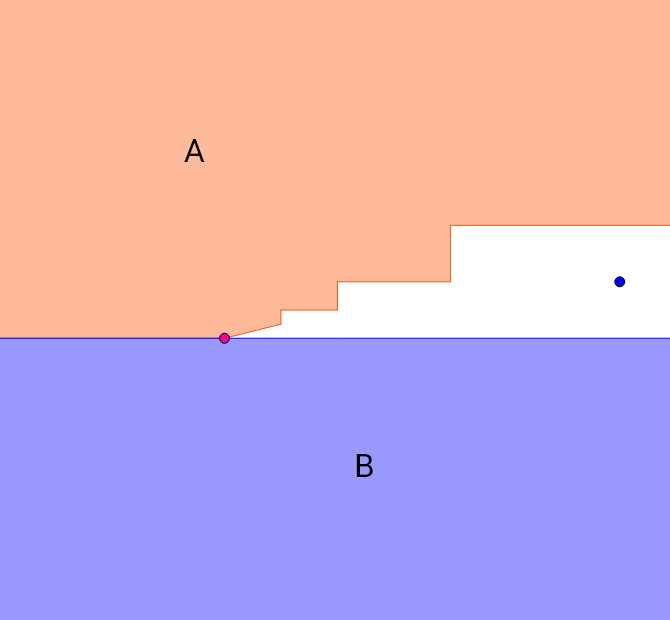}}
\caption{Subtransversality}\label{F2}
\end{figure}
The pair of sets in this example satisfies even the global linear regularity condition \eqref{LR}.
However, the set $A$ is not convex, and the sequence determined by the alternating projections becomes stationary and is not going to converge to any point in $A\cap B$.
At the same time, many important applications naturally lead to feasibility problems for collections of nonconvex sets, and alternating projections often demonstrate reasonably good convergence.

Lewis and Malick \cite{LewMal08} and Lewis et al. \cite{LewLukMal09} demonstrated recently in the Euclidean space setting that the uniform version of the subtransversality property --- called in this article \emph{transversality} (see Definition~\ref{D1}(ii)) --- guarantees local linear convergence of alternating projections for, respectively, a pair of smooth manifolds or a pair of arbitrary closed sets one of which is \emph{super-regular} at the reference point.
Next, Drusvyatskiy et al. \cite{DruIofLew15} showed that the super-regularity assumption can be dropped at the expense of reduced (but still linear) convergence rate.
The transversality property holds, for instance, in the example in Fig.~\ref{F1}(a).
Thanks to \cite{LewMal08,LewLukMal09,DruIofLew15}, the transversality has become a candidate for the position of the core regularity property for a pair of nonconvex sets with nonempty intersection from the point of view of local convergence of alternating projections.

However, the picture in the nonconvex setting is far from being complete.
The transversality is not necessary for the local convergence of alternating projections even in the convex case.
For example, it always fails when the affine span of the union of the sets is not equal to the whole space, while alternating projections can still converge linearly as is the case when the sets are convex with nonempty intersection of their relative interiors.
Another example is given in Fig.~\ref{F2}(a).
Comparing this example with the one in Fig.~\ref{F1}(a) illustrates the difference between the transversality and subtransversality properties.
In these two examples, this difference does not affect the convergence of alternating projections.
The role of the transversality property in the convergence analysis of alternating projections in the nonconvex setting has been further studied in Noll and Rondepierre \cite{NolRon16}, and Kruger et al. \cite{KruTha16,KruLukTha}.

A quest has started for the weakest regularity property lying between transversality and subtransversality and still being sufficient for the local linear convergence of alternating projections in the nonconvex setting.
We mention here the articles by Bauschke et al. \cite{BauLukPhaWan13.1,BauLukPhaWan13.2} utilizing \emph{restricted normal cones}, Drusvyatskiy et al. \cite{DruIofLew15} introducing and successfully employing \emph{intrinsic transversality}, Noll and Rondepierre \cite{NolRon16} introducing a concept of \emph{separable intersection}, with 0-separability being a weaker property than intrinsic transversality and still implying the local linear convergence of alternating projections under the additional assumption that one of the sets is 0-\emph{H\"older regular} at the reference point with respect to the other.
In particular, Drusvyatskiy et al. have shown in \cite[Proposition~3.2 and Theorems~6.1 and 6.2]{DruIofLew15} that, for a pair of closed sets with nonempty intersection, it holds
\begin{align*}
\text{transversality}
\quad\Longrightarrow\quad
\text{intrinsic transversality}
\quad\Longrightarrow\quad
\text{subtransversality},
\end{align*}
and intrinsic transversality ensures local linear convergence of alternating projections.
This makes intrinsic transversality the main candidate for the role of the core regularity property from the
point of view of local convergence of alternating projections.
It is also the main object of interest in the current article.

All the considerations above are for the case when the pair of sets under consideration has nonempty intersection.
At the same time, the alternating projections used in the current article for illustrating the transversality theory of collections of sets can be applied in situations when the intersection is empty, and one can still talk about their `convergence' to some collection of points solving an \emph{inconsistent feasibility} problem.
This motivates expanding the transversality theory to the case of inconsistent feasibility.
The first attempt of this kind has been made recently by Luke et al. \cite{LukThaTam}.

Intrinsic transversality is defined in \cite{DruIofLew15} in the Euclidean space setting using Fr\'echet normal cones.
Unlike intrinsic transversality, the conventional properties of transversality and subtransversality are defined in the setting of an arbitrary normed linear space in purely primal space terms (see Definition~\ref{D1}).
However, in applications it is often more convenient to work with dual space conditions in terms of normal cones.
In the case of transversality, its equivalent Fr\'echet normal cone characterizations in Asplund spaces (or even in general Banach spaces if the sets are convex) are well known.
These representations and not the original primal space definitions were mainly used in \cite{LewMal08,LewLukMal09} when establishing local linear convergence of alternating projections for pairs of nonconvex sets in a finite dimensional space.
For subtransversality, no normal cone conditions have been known up until recently.
The first condition of this type was announced without proof in the Euclidean space setting in \cite[Theorem~4(ii)]{KruLukTha}.
A more general and slightly improved Asplund space version of this result has been proved in \cite{KruLukTha2}.
Unlike the case of transversality, the mentioned normal cone conditions characterizing subtransversality are only sufficient.
The subtransversality property, as is well known, lacks stability.
This fact makes obtaining general necessary and sufficient normal cone characterizations of this property highly unlikely.
The two sets of sufficient normal cone conditions of the subtransversality property, established in \cite{KruLukTha2}, are themselves important transversality/regularity properties of pairs of sets in general normed linear spaces (see Definition~\ref{D3+} below) lying between transversality and subtransversality.
In a finite dimensional Euclidean space, the strongest of the two properties is equivalent (see Theorem~\ref{T6} below) to intrinsic transversality as defined by Drusvyatskiy et al. \cite{DruIofLew15}.
Borrowing partially the terminology from \cite{DruIofLew15}, the two properties are called here \emph{weak intrinsic transversality} and \emph{intrinsic transversality}.

In the current article which continues \cite{KruLukTha2}, the two versions of intrinsic transversality are studied in detail mainly in the finite dimensional setting.
First, spaces with arbitrary norms are considered, and then the results are further specified for Euclidian spaces.

In Section~\ref{s:subtrans} we recall the definitions of \emph{transversality} and \emph{subtransversality} properties of pairs of sets in general normal linear spaces and provide their dual space normal cone necessary and sufficient or just sufficient characterizations in Asplund spaces.
If the sets are convex, the characterizations are formulated in general Banach spaces.
A slightly simpler version of the Asplund space dual sufficient condition of subtransversality from \cite{KruLukTha2} is given.
Then the definitions of \emph{intrinsic transversality} and \emph{weak intrinsic transversality} from \cite{KruLukTha2} are reproduced and their relationships with the conventional \emph{subtransversality} property are formulated.

Sections~\ref{S3} and \ref{S4} are restricted to the finite dimensional situation.
In Section~\ref{S3} two new limiting objects are introduced for pairs of sets: the cone of \emph{pairs of relative limiting normals} and the cone of \emph{pairs of restricted relative limiting normals}.
They allow one to formulate simple limiting criteria of \emph{intrinsic transversality} and in the convex case also \emph{subtransversality}.
The criteria are further simplified if the Euclidian norm is used.
In particular, it is shown that, if the sets are convex, the properties of \emph{intrinsic transversality}, \emph{weak intrinsic transversality} and \emph{subtransversality} are equivalent.
In Section~\ref{S4} several more criteria of intrinsic transversality are presented.
In particular, it is shown that the normed linear space definition of intrinsic transversality adopted in this article, in the Euclidean space setting reduces to the original definition of this property due to Drusvyatskiy et al. \cite{DruIofLew15}.
This justifies the name `intrinsic transversality' used throughout this article.

The concluding Section~\ref{S5} contains a collection of questions related to the content of the article, to which the author does not know the answers.

\paragraph{Notation and preliminaries.}

Given a normed linear space $X$,
its topological dual is denoted by $X^*$, while $\langle\cdot,\cdot\rangle$ denotes the bilinear form defining the pairing between the spaces.
$\B$ and $\B^*$ stand for the closed unit balls in $X$ and $X^*$, respectively, while $\B_\de(x)$ denotes the open ball with centre at $x$ and radius $\de>0$.
Given a set $A$ in a normed linear space, its interior and boundary are denoted by $\Int A$ and $\bd A$, respectively, while $\cone A$ denotes the cone generated by $A$: $\cone A:=\{ta\mid a\in A, t\ge0\}$.
$d_A(x)$ stands for the distance from a point $x$ to a set $A$.
Given an $\al\in\R_\infty:=\R\cup\{+\infty\}$, $\al_+$ denotes its positive part: $\al_+:=\max\{\al,0\}$.
We regularly use the convention that the supremum of the empty subset of $\R_+$ equals~0.
Regarding the infimum of the empty subset of $\R_+$, we occasionally use different conventions which are always explicitly specified in the text: depending on the context, the infimum can be assumed equal either 1 or 2.

Dual characterizations of transversality and subtransversality properties involve dual space objects -- \emph{normal cones}.
For the detailed discussion of the objects introduced below, the readers are referred to the standard references \cite{RocWet98,Mor06.1,Kru03}.
The terminology and notation adopted here mostly follow that in\cite{Kru03}.

Given a subset $A$ of a normed linear space $X$ and a point $\bx\in A$, the \emph{Fr\'echet normal cone} to $A$ at $\bx$ is defined as follows:
\begin{gather}\label{NC1}
N_{A}(\bx):= \left\{x^*\in X^*\mid \limsup_{a\to\bx,\,a\in A\setminus\{\bx\}} \frac {\langle x^*,a-\bx \rangle}{\|a-\bx\|} \leq 0 \right\}.
\end{gather}
It is a nonempty norm closed convex cone, often trivial ($N_{A}(\bx)=\{0\}$).
Similarly,
given a function $f:X\to\R_\infty:=\R\cup\{+\infty\}$ and a point $\bx\in\dom f$, the \emph{Fr\'echet subdifferential} of $f$ at $\bx$ is defined as
\begin{gather}\label{sd}
\sd f(\bx):= \left\{x^*\in X^*\mid \liminf_{x\to\bx,\,x\ne \bx} \frac {f(x)-f(\bx)-\langle x^*,x-\bx \rangle}{\|x-\bx\|} \geq 0 \right\}.
\end{gather}
It is a norm closed convex set, often empty.

If $\dim X<\infty$, the \emph{limiting normal cone} to $A$ at $\bx$
can be useful:
\begin{gather}\label{NC3}
\overline{N}_{A}(\bar x):= \Limsup_{a\to\bx,\,a\in A}N_{A}(a):=\left\{x^*=\lim_{k\to\infty}x^*_k\mid x^*_k\in N_{A}(a_k),\;a_k\in A,\;a_k\to\bx\right\}.
\end{gather}
If $X$ is a Euclidian space and $A$ is closed, the Fr\'echet normal cones in definition \eqref{NC3} can be replaced by the \emph{proximal} ones:
\begin{gather}\label{NC2}
N_{A}^p(\bx):=\cone\left(P_A^{-1}(\bx)-\bx\right).
\end{gather}
Here $P_A$ is the \emph{projection} mapping:
\begin{gather}\label{Pr}
P_{A}(x):=\set{a\in A\mid\|x-a\|=d_A(x)},\quad x\in X.
\end{gather}
If $A$ is closed and convex, then $P_A$ is a singleton.
It is easy to verify that
$N_{A}^p(\bx)\subset N_{A}(\bx)$,
and $\overline{N}_{A}(\bar x)\ne\{0\}$ if and only if $\bx\in\bd A$.
Unlike \eqref{NC1} and \eqref{NC2}, the cone \eqref{NC3} can be nonconvex.

If $A$ is a convex set, then all three cones \eqref{NC1}, \eqref{NC3} and \eqref{NC2} coincide and reduce to the normal cone in the sense of convex analysis:
\begin{gather*}\label{CNC}
N_{A}(\bx):= \left\{x^*\in X^*\mid \langle x^*,a-\bx \rangle \leq 0 \qdtx{for all} a\in A\right\}.
\end{gather*}

Recall that a Banach space is \emph{Asplund} if every continuous convex function on an open convex set is Fr\'echet differentiable on some its dense subset \cite{Phe93}, or equivalently, if the dual of each its separable subspace is separable.
We refer the reader to \cite{Phe93,Mor06.1,BorZhu05} for discussions about and characterizations of Asplund spaces.
All reflexive, in particular, all finite dimensional Banach spaces are Asplund.

\section{Transversality, subtransversality and intrinsic transversality}\label{s:subtrans}

For brevity, in this article we consider the case of two nonempty sets $A$ and $B$.
The extension of the definitions and characterizations of the properties to the case of any finite collection of $n$ sets ($n>1$) is straightforward (cf. \cite{Kru05,Kru06,Kru09,KruTha13,KruTha15}).
The sets are assumed to have a common point $\bar{x}\in A\cap B$.
The notation $\{A,B\}$ is used when referring to the pair of two sets $A$ and $B$ as a single object.

\paragraph{Transversality and subtransversality.}
We first briefly recall two standard regularity properties of a pair of sets in a normed linear space, namely \emph{transversality} and \emph{subtransversality} (also known under other names).

\begin{definition}\label{D1}
Suppose $X$ is a normed linear space, $A,B\subset X$, and $\bx\in A\cap B$.
\begin{enumerate}
\item
$\{A,B\}$ is \emph{subtransversal} at $\bar x$ if
there exist numbers $\alpha\in]0,1[$ and $\delta>0$ such that
\begin{gather}\label{D1-1}
\alpha d\left(x,A\cap B\right)\le \max\left\{d(x,A),d(x,B)\right\}
\quad\mbox{for all}\quad x\in \B_{\delta}(\bar{x}).
\end{gather}
\item
$\{A,B\}$ is \emph{transversal} at $\bar x$ if
there exist numbers $\alpha\in]0,1[$ and $\delta>0$ such that
\begin{multline}\label{D1-2}
\alpha d\left(x,(A-x_1)\cap (B-x_2)\right)\le \max\left\{d(x,A-x_1),d(x,B-x_2)\right\}
\\\mbox{for all}\;\; x\in \B_{\delta}(\bar{x}),\;x_1,x_2\in \delta\B.
\hspace{-.2cm}
\end{multline}
\end{enumerate}
The exact upper bound of all $\alpha\in]0,1[$ such that condition \eqref{D1-1} or condition \eqref{D1-2} is satisfied for some $\de>0$ is denoted by $\str$ or $\tr$, respectively, with the convention that the supremum of the empty set equals~0.
\end{definition}

The requirement that $\al<1$ in both parts of Definition~\ref{D1} imposes no restrictions on the property.
It is only needed in the case $\bx\in\Int(A\cap B)$ (when conditions \eqref{D1-1} and \eqref{D1-2} are satisfied for some $\de>0$ with any $\al>0$) to ensure that $\str$ and $\tr$ are always less than or equal to 1 and simplify the subsequent quantitative estimates.
It is easy to check that when $\bx\in\bd(A\cap B)$, each of the conditions \eqref{D1-1} and \eqref{D1-2} implies $\al\le1$.
We are going to use similar requirements in other definitions throughout the article.

The subtransversality (transversality) of $\{A,B\}$ is equivalent to the condition $\str>0$ ($\tr>0$), and $\str$ ($\tr$) provides a quantitative characterization of this property.

The metric property in part (i) of Definition~\ref{D1} is a very well known regularity property that has been around for more than 30 years under various names ((local) \emph{linear regularity}, \emph{metric regularity}, \emph{linear coherence}, \emph{metric inequality}, and \emph{subtransversality}); cf. \cite{BakDeuLi05,
BauBor93,Dol82,
RocWet98,BauBor96, Iof89,Iof00_,Iof16,KlaLi99,HesLuk13,
LiNgPon07,NgaThe01,Pen13,ZheNg08,ZheWeiYao10, DruIofLew15}.
It has been used as the key assumption when establishing linear convergence of sequences generated by alternating projection algorithms and a qualification condition for subdifferential and normal cone calculus formulae.
If the sets are convex, it is equivalent to the linear regularity property \eqref{LR}.

The property in part (ii) of Definition~\ref{D1} was referred to in \cite{Kru05,Kru06,Kru09} as \emph{strong metric inequality}.
If $A$ and $B$ are closed convex sets and $\Int A\ne\emptyset$, it is equivalent to the conventional qualification condition: $\Int A\cap B\ne\emptyset$ (cf. \cite[Proposition~14]{Kru05}).

There are other equivalent primal space definitions for each of the properties in Definition~\ref{D1}; cf. \cite{Kru05,Kru06,Kru09,KruTha13,KruTha15}.

From comparing the properties in Definitions~\ref{D1}, one can see that the transversality of a pair of sets corresponds to the subtransversality of all their small translations holding uniformly (cf. \cite[p.~1638]{DruIofLew15}).
The next inequality is straightforward:
$$\tr\le\str.$$
We refer the reader to \cite{KruTha15,KruLukTha2} for more examples illustrating the relationship between the properties in Definition~\ref{D1}.

\begin{remark}
1. The maximum of the distances in Definition~\ref{D1} and some other representations in the sequel corresponds to the maximum norm in $\R^2$ employed in all these definitions and assertions.
It can be replaced everywhere by the sum norm (pretty common in this type of definitions in the literature) or any other equivalent norm.
All quantitative characterizations of the properties will remain valid (as long as the same norm is used everywhere), although the exact values of $\str$ and $\tr$ do depend on the chosen norm and some estimates can change.

2. In some situations it can be convenient to use the reciprocals $(\str)\iv$ and $(\tr)\iv$ instead of $\str$ and $\tr$, respectively, when characterizing the corresponding properties.
Instead of checking whether the constant is nonzero when verifying the property, one would need to check wether its reciprocal is finite.
\xqed\end{remark}

Transversality properties of pairs of sets 
are strongly connected with the corresponding regularity properties of \SVM s.
The properties in parts (i) and (ii) of Definition~\ref{D1} correspond, respectively, to \emph{metric subregularity} and \emph{metric regularity} of \SVM s (cf. \cite{Kru05,Kru06,Kru09,KruTha15,KruLukTha,KruLukTha2, Iof00_,Iof16}), which partially explains the terminology adopted in the current article.
These regularity properties of \SVM s lie at the core of the contemporary variational analysis.
They have their roots in classical analysis and are crucial for the study of stability of solutions to (generalized) equations and various aspects of subdifferential calculus and optimization theory.
For the state of the art of the regularity theory of \SVM s and its numerous applications we refer the reader to the book by Dontchev and Rockafellar \cite{DonRoc14} and the comprehensive survey by Ioffe \cite{Iof16,Iof16.2}.


\paragraph{Dual characterizations.}

The dual criterion for the transversality property in Definition~\ref{D1}(ii) in Asplund spaces is well known; see \cite{Kru05,Kru06,Kru09,KruTha13,KruTha15}.

\begin{theorem}\label{T0}
Suppose $X$ is Asplund, $A,B\subset X$ are closed, and $\bx\in A\cap B$. Then
$\{A,B\}$ is transversal at $\bar x$ if and only if there exist numbers $\alpha\in]0,1[$ and $\delta>0$ such that
$\|x^*_1+x^*_2\|>\alpha$ for all $a\in A\cap\B_\delta(\bx)$, $b\in B\cap\B_\delta(\bx)$, and all $x_1^*\in N_{A}(a)$ and $x_2^*\in N_{B}(b)$ satisfying $\|x^*_1\|+\|x^*_2\|=1$.
Moreover, the exact upper bound of all such $\al$ equals $\tr$.
\end{theorem}

In finite dimensions, the above criterion admits convenient equivalent reformulations in terms of limiting normals.

\begin{corollary}\label{C00}
Suppose $\dim X<\infty$, $A,B\subset X$ are closed, and $\bx\in A\cap B$.
The following conditions are equivalent:
\begin{enumerate}
\item
$\{A,B\}$ is transversal at $\bar x$;
\item
there exists a number $\alpha\in]0,1[$ such that ${\|x^*_1+x^*_2\|>\alpha}$
for all $x^*_1\in\overline{N}_{A}(\bar x)$ and $x^*_2\in\overline{N}_{B}(\bar x)$
satisfying
$\|x^*_1\|+\|x^*_2\|=1$;
\item
$\overline{N}_{A}(\bar x)\cap\left(-\overline{N}_{B}(\bar x)\right)=\{0\}$.
\end{enumerate}
Moreover, the exact upper bound of all $\alpha$ in {\rm (ii)} equals $\tr$.
\end{corollary}

The property in part (iii) of Corollary~\ref{C00} is a well known qualification condition/nonseparabilty property that has been around for about 30 years under various names (\emph{basic qualification condition}, \emph{normal qualification condition}, \emph{transversality}, \emph{transversal intersection}, \emph{regular intersection}, \emph{linearly regular intersection}, and \emph{alliedness property}); cf. \cite{Mor88,Mor06.1,ClaLedSteWol98,Pen13,LewMal08, LewLukMal09,Iof16}.
When $A$ and $B$ are smooth manifolds, it coincides with \eqref{1}.

The next two theorems established recently in \cite{KruLukTha2} deal with the subtransversality property in Definition~\ref{D1}(i).
They provide, respectively, a dual sufficient condition for this property in Asplund spaces and a necessary and sufficient dual criterion for convex sets in general Banach spaces.
Not surprisingly, the second statement is simpler.

\begin{theorem}\label{T1}
Suppose $X$ is Asplund, $A,B\subset X$ are closed, and $\bx\in A\cap B$. Then
$\{A,B\}$ is subtransversal at $\bar x$ if there exist numbers $\alpha\in]0,1[$ and $\delta>0$ such that,
for all $a\in(A\setminus B)\cap\B_\de(\bx)$, $b\in(B\setminus A)\cap\B_\de(\bx)$ and $x\in\B_\de(\bx)$ with $\norm{x-a}=\norm{x-b}$,
there exists an $\eps>0$ such that
$\|x^*_1+x^*_2\|>\alpha$ for all $a'\in A\cap\B_\eps(a)$, $b'\in B\cap\B_\eps(b)$, $x_1'\in\B_\eps(a)$, $x_2'\in\B_\eps(b)$, $x'\in\B_\eps(x)$ with $\norm{x'-x_1'}=\norm{x'-x_2'}$, and $x_1^*,x_2^*\in X^*$
satisfying
\begin{gather}\label{T1-2}
\|x^*_1\|+\|x^*_2\|=1,\quad
\ang{x^*_1,x'-x_1'}=\|x^*_1\|\|x'-x_1'\|,\quad
\ang{x^*_2,x'-x_2'}=\|x^*_2\|\|x'-x_2'\|,
\\\label{T1-3}
d(x_1^*,N_{A}(a'))<\delta,\quad d(x_2^*,N_{B}(b'))<\delta.
\end{gather}
Moreover, $\str\ge\al$.
\end{theorem}

\begin{theorem}\label{T2}
Suppose $X$ is a Banach space, $A,B\subset X$ are closed and convex, and $\bx\in A\cap B$.
Then $\{A,B\}$ is subtransversal at $\bar x$ if and only if there exist numbers $\alpha\in]0,1[$ and $\delta>0$ such that $\|x^*_1+x^*_2\|>\alpha$
for all $a\in(A\setminus B)\cap\B_\de(\bx)$, $b\in(B\setminus A)\cap\B_\de(\bx)$, $x\in\B_\de(\bx)$ with $\norm{x-a}=\norm{x-b}$, and $x_1^*,x_2^*\in X^*$ satisfying
\begin{gather}\label{T2-1}
\|x^*_1\|+\|x^*_2\|=1,\quad
\ang{x^*_1,x-a}=\|x^*_1\|\|x-a\|,\quad
\ang{x^*_2,x-b}=\|x^*_2\|\|x-b\|,
\\\notag
d(x_1^*,N_{A}(a))<\delta,\quad d(x_2^*,N_{B}(b))<\delta.
\end{gather}
Moreover, the exact upper bound of all such $\al$ equals $\str$.
\end{theorem}

Below we reformulate Theorem~\ref{T1} in a slightly simpler way (one parameter less).

\begin{theorem}\label{T1'}
Suppose $X$ is Asplund, $A,B\subset X$ are closed, and $\bx\in A\cap B$. Then
$\{A,B\}$ is subtransversal at $\bar x$ if there exist numbers $\alpha\in]0,1[$ and $\delta>0$ such that,
for all $a\in(A\setminus B)\cap\B_\de(\bx)$, $b\in(B\setminus A)\cap\B_\de(\bx)$ and $x\in\B_\de(\bx)$ with $\norm{x-a}=\norm{x-b}$,
there exists an $\eps>0$ such that
$\|x^*_1+x^*_2\|>\alpha$ for all $a'\in A\cap\B_\eps(a)$, $b'\in B\cap\B_\eps(b)$,
$x_1'\in\B_\eps(a)$, $x_2'\in\B_\eps(b)$ with $\norm{x-x_1'}=\norm{x-x_2'}$, and $x_1^*,x_2^*\in X^*$ satisfying \eqref{T1-3} and
\begin{gather}\label{T1-1'}
\|x^*_1\|+\|x^*_2\|=1,\quad
\ang{x^*_1,x-x_1'}=\|x^*_1\|\|x-x_1'\|,\quad
\ang{x^*_2,x-x_2'}=\|x^*_2\|\|x-x_2'\|.
\end{gather}
Moreover, $\str\ge\al$.
\end{theorem}

The conditions in Theorem~\ref{T1} obviously imply those in Theorem~\ref{T1'}.
In fact, the opposite implication is also true, and
Theorem~\ref{T1'} is a consequence of Theorem~\ref{T1}.


\begin{proof}[Theorem~\ref{T1'} from Theorem~\ref{T1}]
Suppose the conditions of Theorem~\ref{T1'} are satisfied with some numbers $\alpha\in]0,1[$ and $\delta>0$.
Take any $a\in(A\setminus B)\cap\B_\de(\bx)$, $b\in(B\setminus A)\cap\B_\de(\bx)$ and $x\in\B_\de(\bx)$ with $\norm{x-a}=\norm{x-b}$, and choose an $\eps>0$ in accordance with the conditions of Theorem~\ref{T1'}.
Next set $\eps':=\eps/2$ and take any $a'\in A\cap\B_{\eps'}(a)$, $b'\in B\cap\B_{\eps'}(b)$, $x_1'\in\B_{\eps'}(a)$, $x_2'\in\B_{\eps'}(b)$, $x'\in\B_{\eps'}(x)$ with $\norm{x'-x_1'}=\norm{x'-x_2'}$, and $x_1^*,x_2^*\in X^*$
satisfying \eqref{T1-2} and \eqref{T1-3}.
Then $x_1'':=x_1'+x-x'\in\B_\eps(a)$, $x_2'':=x_2'+x-x'\in\B_\eps(b)$, $\norm{x-x_1''}=\norm{x-x_2''}$, and conditions \eqref{T1-1'} are satisfied with $x_1''$ and $x_2''$ in place of $x_1'$ and $x_2'$, respectively.
Hence, $\|x^*_1+x^*_2\|>\alpha$, i.e., the conditions of Theorem~\ref{T1} are satisfied with the same numbers $\alpha\in]0,1[$ and $\delta>0$, and $\{A,B\}$ is subtransversal at $\bar x$ with $\str\ge\al$.
\qed\end{proof}

\begin{remark}
1. It is sufficient to check the conditions of  Theorems~\ref{T0}--\ref{T1'} only for $x^*_1\ne0$ and $x^*_2\ne0$.
Indeed, if one of the vectors $x^*_1$ and $x^*_2$ equals 0, then by the normalization condition $\|x^*_1\|+\|x^*_2\|=1$, the norm of the other one equals 1, and consequently $\|x^*_1+x^*_2\|=1$, i.e., such pairs $x^*_1,x^*_2$ do not impose any restrictions on $\al$.

2. Similarly to the classical condition \eqref{1}, the (sub)transversality characterizations in Theorems~\ref{T0}--\ref{T1'} require that among all admissible (i.e., satisfying all the conditions of the theorems) pairs of nonzero elements $x^*_1$ and $x^*_2$ there is no one with $x^*_1$ and $x^*_2$ oppositely directed.

3. The sum $\|x^*_1\|+\|x^*_2\|$ in Theorems~\ref{T0}--\ref{T1'} corresponds to the sum norm on $\R^2$, which is dual to the maximum norm on $\R^2$ used in Definition~\ref{D1}.
It can be replaced by $\max\{\|x^*_1\|,\|x^*_2\|\}$ (cf. \cite[(6.11)]{Pen13}) or any other norm on $\R^2$.
\xqed\end{remark}

The proof of Theorems~\ref{T1} and \ref{T2} given in \cite{KruLukTha2} follows the sequence proposed in \cite{Kru15} when deducing metric subregularity criteria for \SVM s and consists of a series of propositions providing lower primal and dual estimates for the constant $\str$ and, thus, sufficient conditions for the subtransversality of the pair $\{A,B\}$ at $\bx$ which can be of independent interest.
In what follows, we will use notations $\itrd{w}$ and $\strc$ for the supremum of all $\al$ in Theorems~\ref{T1'} and \ref{T2}, respectively, with the convention that the supremum over the empty set equals 0.
It is easy to check the following explicit representations of the two constants:

\begin{align}\notag
\itrd{w}:=& \lim_{\rho\downarrow0} \inf_{\substack{a\in(A\setminus B)\cap\B_\rho(\bx),\;b\in(B\setminus A)\cap\B_\rho(\bx)\\x\in\B_\rho(\bx),\; \norm{x-a}=\norm{x-b}}}
\\\label{itrw}
&\liminf_{\substack{x_1'\to a,\;x_2'\to b,\; a'\to a,\;b'\to b\\a'\in A,\;b'\in B,\; \|x-x_1'\|=\|x-x_2'\|\\ d(x_1^*,N_{A}(a'))<\rho,\; d(x_2^*,N_{B}(b'))<\rho,\; \|x_1^*\|+\|x_2^*\|=1
\\
\langle x_1^*,x-x_1'\rangle=\|x_1^*\|\,\|x-x_1'\|,\;
\langle x_2^*,x-x_2'\rangle=\|x_2^*\|\,\|x-x_2'\|}}
\|x^*_1+x^*_2\|,
\\\label{itrc}
\strc:=& \liminf_{\substack{x\to\bx,\;a\to\bx,\;b\to\bx\\a\in A\setminus B,\;b\in B\setminus A,\; \norm{x-a}=\norm{x-b}\\ d(x_1^*,N_{A}(a))\to0,\;d(x_2^*,N_{B}(b))\to0,\; \|x_1^*\|+\|x_2^*\|=1
\\
\langle x_1^*,x-a\rangle=\|x_1^*\|\,\|x-a\|,\;
\langle x_2^*,x-b\rangle=\|x_2^*\|\,\|x-b\|}}
\|x^*_1+x^*_2\|,
\end{align}
with the convention that the infimum over the empty set equals 1.

\paragraph{Intrinsic transversality.}

The two-limit definition \eqref{itrw} as well as the corresponding dual space sufficient characterization of subtransversality in Theorem~\ref{T1'} look complicated and difficult to verify.
The following one-limit modification of \eqref{itrw} in terms of Fr\'echet normals can be useful:
\begin{gather}\label{itr}
\itr:=\liminf\limits_{\substack{
a\to\bar x,\; b\to\bar x,\; x\to\bar x\\
a\in A\setminus B,\;b\in B\setminus A,\;x\ne a,\;x\ne b\\ x^*_1\in N_{A}(a)\setminus\{0\},\;x^*_2\in N_{B}(b)\setminus\{0\},\;\norm{x^*_1}+\norm{x^*_2}=1
\\
\frac{\norm{x-a}}{\norm{x-b}}\to1,\; \frac{\ang{x^*_1,x-a}}{\norm{x^*_1}\norm{x-a}}\to1,\; \frac{\ang{x^*_2,x-b}}{\norm{x^*_2}\norm{x-b}}\to1}}
\|x^*_1+x^*_2\|,
\end{gather}
with the convention that the infimum over the empty set equals 1.

The relationships between the constants $\str$, $\strc$, $\itr$ and $\itrd{w}$ are collected in the next proposition.

\begin{proposition}\label{P10}
Suppose $X$ is a Banach space, $A,B\subset X$ are closed, and $\bx\in A\cap B$.
\begin{enumerate}
\item
$0\le\itr\le\itrd{w}\le\strc\le1$;
\item
if $X$ is Asplund, then $\str\ge\itrd{w}$;
\item
if $\dim X<\infty$, then
\begin{gather}\label{itr"}
\itrd{w}= \liminf_{\substack{a\to\bx,\;b\to\bx,\; x\to\bx\\a\in A\setminus B,\;b\in B\setminus A,\; \norm{x-a}=\norm{x-b}\\ d(x_1^*,\overline{N}_{A}(a))\to0,\; d(x_2^*,\overline{N}_{B}(b))\to0,\; \|x_1^*\|+\|x_2^*\|=1
\\
\langle x_1^*,x-a\rangle=\|x_1^*\|\,\|x-a\|,\;
\langle x_2^*,x-b\rangle=\|x_2^*\|\,\|x-b\|}}
\|x^*_1+x^*_2\|,
\end{gather}
with the convention that the infimum over the empty set equals 1;
\item
if $A$ and $B$ are convex, then $\str=\strc$;
\item
if $\dim X<\infty$, and $A$ and $B$ are convex, then $\itrd{w}=\strc=\str$.
\end{enumerate}
\end{proposition}

\begin{proof}
Part (i) follows immediately from the definitions.
Parts (ii) and (iv) are consequences of Theorems~\ref{T1'} and \ref{T2}, respectively.
Parts (iii) and (v) have been established in \cite{KruLukTha2}.
\qed\end{proof}

The property introduced in Theorem~\ref{T1} (or equivalently, Theorem~\ref{T1'}) as a sufficient dual space characterization of subtransversality and corresponding to the condition $\itrd{w}>0$ as well as the stronger property corresponding to the condition $\itr>0$ are themselves important transversality properties of the pair $\{A,B\}$ at $\bx$.
Borrowing partially the terminology from \cite{DruIofLew15}, these properties are referred to in \cite{KruLukTha2} as \emph{weak intrinsic transversality} and \emph{intrinsic transversality}, respectively.
\begin{definition}\label{D3+}
Suppose $X$ is a normed linear space, $A,B\subset X$ are closed, and $\bx\in A\cap B$.
\begin{enumerate}
\item
$\{A,B\}$ is
weakly intrinsically transversal at $\bar x$ if $\itrd{w}>0$, i.e., there exist numbers $\alpha\in]0,1[$ and $\delta>0$ such that, for all $a\in(A\setminus B)\cap\B_\de(\bx)$, $b\in(B\setminus A)\cap\B_\de(\bx)$ and $x\in\B_\de(\bx)$
with $\norm{x-a}=\norm{x-b}$,
one has $\|x^*_1+x^*_2\|>\alpha$ for some $\eps>0$ and all $a'\in A\cap\B_\eps(a)$, $b'\in B\cap\B_\eps(b)$, $x_1'\in\B_\eps(a)$, $x_2'\in\B_\eps(b)$ with $\norm{x-x_1'}=\norm{x-x_2'}$, and $x_1^*,x_2^*\in X^*$ satisfying conditions \eqref{T1-3} and \eqref{T1-1'};
\item
$\{A,B\}$ is
intrinsically transversal at $\bar x$ if $\itr>0$, i.e., there exist numbers $\alpha\in]0,1[$ and $\delta>0$ such that $\|x^*_1+x^*_2\|>\alpha$
for all $a\in(A\setminus B)\cap\B_\de(\bx)$, $b\in(B\setminus A)\cap\B_\de(\bx)$, $x\in\B_\de(\bx)$ with
$x\ne a$, $x\ne b$, $1-\delta<\frac{\norm{x-a}}{\norm{x-b}}<1+\delta$,
and $x_1^*\in N_{A}(a)\setminus\{0\}$, $x_2^*\in N_{B}(b)\setminus\{0\}$ satisfying
\begin{gather*}
\norm{x^*_1}+\norm{x^*_2}=1,\quad\frac{\ang{x_1^*,x-a}} {\|x_1^*\|\|x-a\|} >1-\delta,\quad
\frac{\ang{x_2^*,x-b}} {\|x_2^*\|\|x-b\|} >1-\delta.
\end{gather*}
\end{enumerate}
\end{definition}

\begin{remark}\label{R7}
1. The properties introduced in Definition~\ref{D3+} are less restrictive than the dual criterion of transversality in Theorem~\ref{T0}.

2. Unlike the transversality and subtransversality properties defined originally by the primal space Definition~\ref{D1} with the dual space characterizations (not always equivalent!) given by Theorems~\ref{T0}--\ref{T1'}, the intrinsic transversality and weak intrinsic transversality properties are defined in Definition~\ref{D3+} directly in dual space terms and do not have in general equivalent primal space representations.
\xqed\end{remark}

In view of Definition~\ref{D3+}, Theorem~\ref{T1'} says that in Asplund spaces weak intrinsic transversality (and consequently intrinsic transversality) implies subtransversality.
Thanks to Proposition~\ref{P10}(i) and (iii), and Remark~\ref{R7}, we have the following relationships between the transversality properties in Asplund spaces.

\begin{corollary}
Suppose $X$ is Asplund, $A,B\subset X$ are closed, and $\bx\in A\cap B$.
Consider the following conditions:
\begin{enumerate}
\item
$\{A,B\}$ is transversal at $\bar x$;
\item
$\{A,B\}$ is intrinsically transversal at $\bar x$;
\item
$\{A,B\}$ is weakly intrinsically transversal at $\bar x$;
\item
$\{A,B\}$ is subtransversal at $\bar x$.
\end{enumerate}
Then {\rm (i)} $\Rightarrow$ {\rm (ii)} $\Rightarrow$ {\rm (iii)} $\Rightarrow$ {\rm (iv)}.
If $\dim X<\infty$, and $A$ and $B$ are convex, then {\rm (iii)} $\Leftrightarrow$ {\rm (iv)}.
\end{corollary}


\section{Intrinsic transversality and relative limiting normals}\label{S3}
From now on we assume that $\dim X<\infty$.
\paragraph{Intrinsic transversality in finite dimensions.}
Definition~\ref{D3+} introduces certain limiting processes (cf. definitions \eqref{itrw}, \eqref{itrc} and \eqref{itr} and representation \eqref{itr"}) and can lead naturally to employing certain limiting normals to the sets under consideration.
Observe that not all limiting normals are relevant for characterizing the intrinsic transversality and weak intrinsic transversality properties of a pair of sets.
Only those normals to each of the sets can be of interest which take into account the relative location of the other set.
It makes sense considering pairs of normals approximately `directed' towards the same point.
This observation motivates considering pairs of relative limiting normals.

\begin{definition}\label{D5}
Suppose $A,B\subset X$ and $\bx\in A\cap B$.
\begin{enumerate}
\item
A pair $(x^*_1,x^*_2)\in X^*\times X^*$ is called
a \emph{pair of relative limiting normals} to $\{A,B\}$ at $\bar x$ if there exist sequences $(a_k)\subset A\setminus B$, $(b_k)\subset B\setminus A$, $(x_k)\subset X$ and $(x_{1k}^*),(x_{2k}^*)\subset X^*$ such that $x_k\ne a_k$, $x_k\ne b_k$ $(k=1,2,\ldots)$, $a_k\to\bx$, $b_k\to\bx$, $x_k\to\bx$, $x_{1k}^*\to x_1^*$, $x_{2k}^*\to x_2^*$, and
\begin{gather*}
x_{1k}^*\in N_A(a_k),\quad x_{2k}^*\in N_B(b_k)\quad (k=1,2,\ldots),
\\
\frac{\norm{x_k-a_k}}{\norm{x_k-b_k}} \to1,\quad
\frac{\ang{x_{1k}^*,x_k-a_k}}{\norm{x_{1k}^*}\norm{x_k-a_k}} \to1,\quad
\frac{\ang{x_{2k}^*,x_k-b_k}}{\norm{x_{2k}^*}\norm{x_k-b_k}} \to1,
\end{gather*}
with the convention that $\frac{0}{0}=1$.
The collections of all pairs of relative limiting normals to $\{A,B\}$ at $\bar x$ will be denoted by $\overline{N}_{A,B}(\bar x)$.
\item
A pair $(x^*_1,x^*_2)\in X^*\times X^*$ is called
a \emph{pair of restricted relative limiting normals} to $\{A,B\}$ at $\bar x$ if there exist sequences $(a_k)\subset A\setminus B$, $(b_k)\subset B\setminus A$, $(x_k)\subset X$ and $(x_{1k}^*),(x_{2k}^*)\subset X^*$ such that $\norm{x_k-a_k}=\norm{x_k-b_k}$ $(k=1,2,\ldots)$,
$a_k\to\bx$, $b_k\to\bx$, $x_k\to\bx$, $x_{1k}^*\to x_1^*$, $x_{2k}^*\to x_2^*$, and
\begin{gather*}
d(x_{1k}^*,N_A(a_k))\to0,\quad d(x_{2k}^*,N_B(b_k))\to0,
\\
\ang{x_{1k}^*,x_k-a_k}=\norm{x_{1k}^*}\norm{x_k-a_k},\quad
\ang{x_{2k}^*,x_k-b_k}=\norm{x_{2k}^*}\norm{x_k-b_k}\quad (k=1,2,\ldots).
\end{gather*}
The collections of all pairs of restricted relative limiting normals to $\{A,B\}$ at $\bar x$ will be denoted by $\overline{N}{}^c_{A,B}(\bar x)$.
\end{enumerate}
\end{definition}

Thus, $\overline{N}_{A,B}(\bar x)$ and $\overline{N}{}^c_{A,B}(\bar x)$ are formed by limits of certain sequences of pairs of Fr\'echet normals to each of the sets `directed' approximately towards the same point.

\begin{remark}\label{D3}
1. In Definition~\ref{D5}, one can always assume that $\norm{x^*_{1k}}=\norm{x^*_1}$, $\norm{x^*_{2k}}=\norm{x^*_2}$, $(k=1,2,\ldots)$.
Indeed, if e.g. $x^*_1=0$, one can take $x^*_{1k}:=0$ $(k=1,2,\ldots)$; if $x^*_1\ne0$, then, without loss of generality, $x^*_{1k}\ne0$ $(k=1,2,\ldots)$, and one can substitute $x^*_{1k}$ with $(x^*_{1k})':=\frac{\norm{x^*_1}}{\norm{x^*_{1k}}}{x^*_{1k}}$.
The same argument applies to $x^*_2$ and $(x^*_{2k})$.

2. Given a subset $A\subset X$, a point $\bx\in A$, and a sequence $(x_k)\subset X$ converging to $\bx$,
it could make sense considering
the set $\overline{N}_{A}(\bar x;(x_k))$ of \emph{limiting normals} to $A$ at $\bx$ relative to $(x_k)$ defined as the set of
vectors $x^*\in X^*$ such that there exist sequences $(a_k)\subset A$ and $(x_{k}^*)\subset X^*$ such that $a_k\ne x_k$ $(k=1,2,\ldots)$, $a_k\to\bx$, $x_{k}^*\to x^*$ as $k\to\infty$ and
\begin{gather*}
x_{k}^*\in N_A(a_k)\quad (k=1,2,\ldots),\quad
\frac{\ang{x_{k}^*,x_k-a_k}}{\norm{x_k^*}\norm{x_k-a_k}} \to1,
\end{gather*}
with the convention that $\frac{0}{0}=1$.

This definition is an important ingredient of Definition~\ref{D5}(i) above.
If $x^*\in\overline{N}_{A}(\bar x;(x_k))$ and $(a_k)\subset A$ is a sequence corresponding to $x^*$ in accordance with this definition, then one has $\ang{x^*,x}=\norm{x^*}$ for any limiting point $x$ of the sequence $\left(\frac{x_k-a_k}{\norm{x_k-a_k}}\right)$.
Obviously, $\overline{N}_{A}(\bar x;(x_k))$ is a cone in $X^*$, and
\begin{gather}\label{84}
\overline{N}_{A}(\bar x;(x_k))\subset \overline{N}_{A}(\bar x).
\end{gather}
Since $\dim X<\infty$, it is easy to check that the cone $\overline{N}_{A}(\bar x;(x_k))$ is closed.
$\overline{N}_{A}(\bar x;(x_k))$ can be empty.
Indeed, if e.g., $A=\{\bx\}$ and $x_k=\bx$ $(k=1,2,\ldots)$,
then there is no sequence $(a_k)\subset A$ with $a_k\ne x_k$, and consequently $\overline{N}_{A}(\bar x;(x_k))=
\emptyset$.
\xqed\end{remark}



\begin{proposition}\label{P7++}
Suppose $A,B\subset X$ and $\bx\in A\cap B$.
\begin{enumerate}
\item
Each of the sets $\overline{N}_{A,B}(\bar x)$ and $\overline{N}{}^c_{A,B}(\bar x)$ is a closed cone in $X^*\times X^*$, possibly empty.
Moreover, if the set contains a pair $(x^*_1,x^*_2)$, then, it also contains the pairs $(t_1 x^*_1,t_2 x^*_2)$ for all $t_1>0$ and $t_2>0$.
\item
$\overline{N}{}^c_{A,B}(\bar x) \subset \overline{N}_{A,B}(\bar x) \subset \bigcup\limits_{\substack{(x_k)\to \bar{x}}} \overline{N}_{A}(\bar x;(x_k)) \times \overline{N}_{B}(\bar x;(x_k))\subset \overline{N}_{A}(\bar x) \times \overline{N}_{B}(\bar x)$.
\end{enumerate}
\end{proposition}
\begin{proof}
(i) We start with the `moreover' assertion.
If a pair $(x^*_1,x^*_2)$ belongs to either of the sets $\overline{N}_{A,B}(\bar x)$ and $\overline{N}{}^c_{A,B}(\bar x)$, and $(a_k)\subset A\setminus B$, $(b_k)\subset B\setminus A$, $(x_k)\subset X$ and $(x_{1k}^*),(x_{2k}^*)\subset X^*$ are the corresponding sequences from Definition~\ref{D5}, then
it is straightforward from Definition~\ref{D5} that, for any $t_1>0$ and $t_2>0$, the sequences $(a_k)$, $(b_k)$, $(x_k)$, $(t_1 x_{1k}^*)$ and $(t_2 x_{2k}^*)$ also satisfy all the conditions in the corresponding part of Definition~\ref{D5}.
Hence, the pair $(t_1 x^*_1,t_2 x^*_2)$ also belongs to the respective set $\overline{N}_{A,B}(\bar x)$ or $\overline{N}{}^c_{A,B}(\bar x)$.
In particular, taking $t_1=t_2$, we conclude that both the sets are cones.

If a sequence of pairs $(x^*_{1i},x^*_{2i})$ belongs to either of the sets $\overline{N}_{A,B}(\bar x)$ and $\overline{N}{}^c_{A,B}(\bar x)$ and the sequences $(x^*_{1i})$ and $(x^*_{2i})$ converge to $x_1^*$ and $x_2^*$, respectively, then, based on the corresponding sequences from Definition~\ref{D5} and using the standard `diagonal' procedure, one can easily construct new sequences satisfying all the conditions in the corresponding part of Definition~\ref{D5}, and the sequences in $X^*$ converge to $x_1^*$ and $x_2^*$.
Thus, the two cones making each of the sets $\overline{N}_{A,B}(\bar x)$ and $\overline{N}{}^c_{A,B}(\bar x)$ are closed.

(ii) All the inclusions are direct consequences of Definition~\ref{D5} and the one in Remark~\ref{D3}.2.
The limiting procedure employed in part (ii) of Definition~\ref{D5} is more restrictive than the one in part (i).
This observation
implies the first inclusion.
If $(x^*_1,x^*_2)\in\overline{N}_{A,B}(\bar x)$, then $x^*_1\in\overline{N}_{A}(\bar x;(x_k))$ and $x^*_2\in\overline{N}_{B}(\bar x;(x_k))$ for some
sequence $(x_k)\subset X$ converging to $\bx$.
Hence, the second inclusion.
The last inclusion is a consequence of the observation \eqref{84}.
\qed\end{proof}

The next example shows that the last two inclusions in Proposition~\ref{P7++}(ii) can be strict.

\begin{example}
Let $X=\mathbb{R}^2$ with the Euclidean norm, $A= \{(t,0)\mid t\ge 0\}$, $B=\{(t,t)\mid t\ge 0\}$
and $\bar{x}=(0,0)$.

Set $a_k=b_k:=\bx$, $x_k:=\left(-\frac{1}{k},-\frac{1}{k}\right)$.
We obviously have $x_k\to\bx$ as $k\to\infty$, $x^*:=(-1,-1)\in N_A(\bx)\cap N_B(\bx)$, and $x_k-\bx=\frac{1}{k}x^*$, i.e., vector $x^*$ is parallel to $x_k-\bx$.
Hence, all the conditions in the definition in Remark~\ref{D3}.2 are satisfied for each of the sets $A$ and $B$, and consequently, $x^*\in\overline{N}_{A}(\bar x;(x_k))\cap\overline{N}_{B}(\bar x;(x_k))$.
However, $(x^*,x^*)\notin\overline{N}_{A,B}(\bar x)$ because it is not possible to satisfy the conditions in Definition~\ref{D5}(i) with the pair $(x^*,x^*)$ and any $a_k\ne\bx$ and $b_k\ne\bx$.

With $x_1^*:=(0,1)\in{N}_{A}(\bar x) =\overline{N}_{A}(\bar x)$ and $x_2^*:=(-1,0)\in{N}_{B}(\bar x) =\overline{N}_{B}(\bar x)$ it is not possible to find a single sequence $(x_k)$ converging to $\bx$ to satisfy all the conditions in Definition~\ref{D5}(i).
Hence, $(x_1^*,x_2^*)\notin\bigcup\limits_{\substack{(x_k)\to \bar{x}}} \overline{N}_{A}(\bar x;(x_k)) \times \overline{N}_{B}(\bar x;(x_k))$.
\xqed
\end{example}

Thanks to Definition~\ref{D5}, definitions \eqref{itr} and \eqref{itrc} admit simpler representations:
\begin{gather}\label{itr+}
\itr= \min_{\substack{(x^*_1,x^*_2)\in\overline{N}_{A,B}(\bar x)\\
\|x^*_1\|+\|x^*_2\|=1}}\|x^*_1+x^*_2\|,
\\\label{itrc+}
\strc= \min_{\substack{(x^*_1,x^*_2)\in\overline{N}{}^c_{A,B}(\bar x)\\
\|x^*_1\|+\|x^*_2\|=1}}\|x^*_1+x^*_2\|,
\end{gather}
with the convention that the minimum over the empty set equals $1$.

\begin{remark}\label{R6}
1. Formulae \eqref{itr+} and \eqref{itrc+} take into account that in finite dimensions the sets under both minima are compact, and the minima of $\|x^*_1+x^*_2\|$ over these sets are attained.

2. It is immediate from \eqref{itr+} and \eqref{itrc+} that conditions under $\min$ there can be complemented by the inequalities $x^*_1\ne0$ and $x^*_2\ne0$.
\xqed\end{remark}

The following limiting criteria of intrinsic transversality in finite dimensions are straightforward.

\begin{theorem}\label{T3}
Suppose $A,B\subset X$ are closed and $\bx\in A\cap B$. The following conditions are equivalent:
\begin{enumerate}
\item
$\{A,B\}$ is intrinsically transversal at $\bar x$;
\item
there exists a number $\alpha\in]0,1[$ such that ${\|x^*_1+x^*_2\|>\alpha}$
for all $(x^*_1,x^*_2)\in\overline{N}_{A,B}(\bar x)$
satisfying
$\|x^*_1\|+\|x^*_2\|=1$;
\item
$\left\{x^*\in X^*\mid (x^*,-x^*)\in\overline{N}_{A,B}(\bar x)\right\} \subset \{0\}$.
\end{enumerate}
Moreover, the exact upper bound of all $\alpha$ in {\rm (ii)} equals $\itr$.
\end{theorem}

\begin{proof}
The equivalence of (i) and (ii) as well as the `moreover' estimate are immediate from comparing Definitions~\ref{D3+}(ii) and \ref{D5}(i).

If (i) does not hold, i.e., $\itr=0$, then there exits a pair $(x^*_1,x^*_2)\in\overline{N}_{A,B}(\bar x)$ such that $\|x^*_1\|+\|x^*_2\|=1$ and $\|x^*_1+x^*_2\|=0$; this violates (iii).
Conversely, if (iii) is violated, i.e., there exists an $x^*\ne0$ such that $(x^*,-x^*)\in\overline{N}_{A,B}(\bar x)$, then the pair $(x^*_1,x^*_2)\in \overline{N}_{A,B}(\bar x)$ with $x^*_1=\frac{x^*}{2\norm{x^*}}$ and $x^*_2=-\frac{x^*}{2\norm{x^*}}$ satisfies $\|x^*_1\|+\|x^*_2\|=1$ and $\|x^*_1+x^*_2\|=0$, which yields $\itr=0$, i.e., (i) does not hold.
\qed\end{proof}


\begin{theorem}\label{T3+}
Suppose $A,B\subset X$ are closed and convex, and $\bx\in A\cap B$.
The following conditions are equivalent:
\begin{enumerate}
\item
$\{A,B\}$ is subtransversal at $\bar x$;
\item
there exists a number $\alpha\in]0,1[$ such that ${\|x^*_1+x^*_2\|>\alpha}$
for all $(x^*_1,x^*_2)\in\overline{N}{}^c_{A,B}(\bar x)$
satisfying
$\|x^*_1\|+\|x^*_2\|=1$;
\item
$\left\{x^*\in X^*\mid (x^*,-x^*)\in\overline{N}{}^c_{A,B}(\bar x)\right\} \subset \{0\}$.
\end{enumerate}
Moreover, the exact upper bound of all $\alpha$ in {\rm (ii)} equals $\str=\strc$.
\end{theorem}

\paragraph{Intrinsic transversality in Euclidean spaces.}
From now on we assume that $X$ is equipped with the Euclidean norm.
We will identify $X^*$ with $X$, use $\ang{\cdot,\cdot}$ to denote the scalar product, and write $v_1$, $v_2$,\ldots\ instead of $x^*_1$, $x^*_2$,\ldots
We start with formulating several technical lemmas which are used in the proofs of the results in this section.
They are consequences of the geometry of Euclidean space and are likely to be well known.
As the author has not been able to find proper references, short proofs are provided for completeness.

\begin{lemma}\label{L2}
Let $(u_k)$ and $(v_k)$ be sequences in a Euclidean space. The following two conditions are equivalent:
\begin{enumerate}
\item
$\ang{u_{k},v_k}-\norm{u_{k}}\norm{v_k}\to0$;
\item
$\norm{v_k}u_{k}-\norm{u_{k}}v_k\to0$.
\end{enumerate}
\end{lemma}
\begin{proof}
Observe that
\begin{align*}
\big\|\norm{v_k}u_{k}-\norm{u_{k}}v_k\big\|^2 &=2(\norm{u_{k}}\norm{v_k})^2 -2\norm{u_{k}}\norm{v_k}\ang{u_{k},v_k}
\\
&=2\norm{u_{k}}\norm{v_k}(\norm{u_{k}}\norm{v_k} -\ang{u_{k},v_k}).
\end{align*}
The equivalence of the two conditions follows.
\qed\end{proof}

\begin{lemma}\label{Tech}
Let $(u_k)$ and $(v_k)$ be sequences of nonzero vectors in a Euclidean space.
If both $\left(\frac{u_k}{\|u_k\|}\right)$ and $\left(\frac{v_k}{\|v_k\|}\right)$ converge to a (unit) vector $u$, then the sequence $\left(\frac{u_k+v_k}{\|u_k+v_k\|}\right)$ also converges to $u$.
\end{lemma}
\begin{proof}
Let
$$
\frac{u_k}{\|u_k\|}\to u
\qdtx{and}
\frac{v_k}{\|v_k\|}\to u.
$$
Then $u_k+v_k\ne0$ for all sufficiently large $k$, because otherwise the first condition above yields $$\frac{v_k}{\|v_k\|}=-\frac{u_k}{\|u_k\|}\to-u,$$
which contradicts the second condition.
Thus,
$$\lim\limits_{k\to\infty}\frac{u_k+v_k}{\|u_k+v_k\|} =\lim\limits_{k\to\infty} \frac{\frac{u_k}{\|u_k\|} +\frac{v_k}{\|v_k\|}\frac{\|v_k\|}{\|u_k\|}}{\norm{\frac{u_k}{\|u_k\|} +\frac{v_k}{\|v_k\|}\frac{\|v_k\|}{\|u_k\|}}} =\lim\limits_{k\to\infty} \frac{u\left(1 +\frac{\|v_k\|}{\|u_k\|}\right)}{\norm{u}\left(1 +\frac{\|v_k\|}{\|u_k\|}\right)}=u.$$
\qed\end{proof}

\begin{lemma}\label{L3}
Let $u$ and $v$ be nonzero vectors in a Euclidean space.
Then
\begin{gather*}
\frac{\norm{u+v}}{\norm{u}+\norm{v}}\ge \frac{1}{2}\norm{\frac{u}{\norm{u}}+ \frac{v}{\norm{v}}}.
\end{gather*}
\end{lemma}

The idea of the proof below originates in the proof of \cite[Proposition~5]{KruTha13}.
\begin{proof}
Let $u,v\in X\setminus\{0\}$.
\begin{align*}
\left(\frac{\norm{u+v}}{\norm{u}+\norm{v}}\right)^2 &=\frac{\norm{u}^2+\norm{v}^2+2\ang{u,v}} {(\norm{u}+\norm{v})^2}
\\
&=\frac{\frac{1}{2} \left((\norm{u}+\norm{v})^2 +(\norm{u}-\norm{v})^2\right) +2\ang{u,v}} {(\norm{u}+\norm{v})^2}
\\
&=\frac{1}{2} \left(1+\frac{(\norm{u}-\norm{v})^2 +4\ang{u,v}} {(\norm{u}+\norm{v})^2}\right)
\\
&=\frac{1}{2} \left(1+\frac{\ang{u,v}}{\norm{u}\norm{v}} +\frac{(\norm{u}-\norm{v})^2 +\left(4-\frac{(\norm{u}+\norm{v})^2} {\norm{u}\norm{v}}\right)\ang{u,v}} {(\norm{u}+\norm{v})^2} \right)
\\
&=\frac{1}{2} \left(1+\frac{\ang{u,v}}{\norm{u}\norm{v}} +\frac{(\norm{u}-\norm{v})^2 -\frac{(\norm{u}-\norm{v})^2} {\norm{u}\norm{v}}\ang{u,v}} {(\norm{u}+\norm{v})^2} \right)
\\
&=\frac{1}{2} \left(1+\frac{\ang{u,v}}{\norm{u}\norm{v}} +\left(\frac{\norm{u}-\norm{v}} {\norm{u}+\norm{v}}\right)^2 \left(1-\frac{\ang{u,v}} {\norm{u}\norm{v}}\right)\right)
\\
&\ge\frac{1}{2} \left(1+\frac{\ang{u,v}}{\norm{u}\norm{v}} \right)=\frac{1}{4} \left(2+2\ang{\frac{u}{\norm{u}},\frac{v}{\norm{v}}} \right)=\frac{1}{4}\norm{\frac{u}{\norm{u}}+ \frac{v}{\norm{v}}}^2.
\end{align*}
The proof is completed.
\qed\end{proof}

To simplify the comparison of various conditions in the rest of the article, we first reformulate Definition~\ref{D5} using the Euclidean space notation stipulated above.

\begin{definition}\label{D5'}
Suppose $X$ is a Euclidean space, $A,B\subset X$, and $\bx\in A\cap B$.
\begin{enumerate}
\item
A pair $(v_1,v_2)\in X\times X$ is
called a \emph{pair of relative limiting normals} to $\{A,B\}$ at $\bar x$, i.e. $(v_1,v_2)\in\overline{N}_{A,B}(\bar x)$, if and only if there exist sequences $(a_k)\subset A\setminus B$, $(b_k)\subset B\setminus A$, $(x_k),(v_{1k}),(v_{2k})\subset X$ such that $x_k\ne a_k$, $x_k\ne b_k$ $(k=1,2,\ldots)$, $a_k\to\bx$, $b_k\to\bx$, $x_k\to\bx$, $\frac{\norm{x_k-a_k}}{\norm{x_k-b_k}} \to1$, $v_{1k}\to v_1$, $v_{2k}\to v_2$ as $k\to\infty$, and
\begin{gather}\label{D5'-3}
v_{1k}\in N_A(a_k),\; v_{2k}\in N_B(b_k)\; (k=1,2,\ldots),
\;
\frac{\ang{v_{1k},x_k-a_k}}{\norm{v_{1k}}\norm{x_k-a_k}} \to1,\;
\frac{\ang{v_{2k},x_k-b_k}}{\norm{v_{2k}}\norm{x_k-b_k}} \to1,
\end{gather}
with the convention that $\frac{0}{0}=1$,
\item
A pair $(v_1,v_2)\in X\times X$ is called
a \emph{pair of restricted relative limiting normals}
to $\{A,B\}$ at $\bar x$, i.e., $(v_1,v_2)\in\overline{N}{}^c_{A,B}(\bar x)$, if and only if there exist sequences $(a_k)\subset A\setminus B$, $(b_k)\subset B\setminus A$, $(x_k),(v_{1k}),(v_{2k})\subset X$ such that $\norm{x_k-a_k}=\norm{x_k-b_k}$ $(k=1,2,\ldots)$, $a_k\to\bx$, $b_k\to\bx$, $x_k\to\bx$, $v_{1k}\to v_1$, $v_{2k}\to v_2$ as $k\to\infty$, and
\begin{gather}\label{D5'-1}
d(v_{1k},N_A(a_k))\to0,\quad d(v_{2k},N_B(b_k))\to0, \\\label{D5'-2}
\ang{v_{1k},x_k-a_k}=\norm{v_{1k}}\norm{x_k-a_k}, \quad
\ang{v_{2k},x_k-b_k}=\norm{v_{2k}}\norm{x_k-b_k} \quad
(k=1,2,\ldots).
\end{gather}
\end{enumerate}
\end{definition}

In the Euclidean setting, Definition~\ref{D5'}(i) admits several equivalent formulations.

\begin{proposition}\label{P15}
Suppose $X$ is a Euclidean space, $A,B\subset X$, and $\bx\in A\cap B$.
\begin{enumerate}
\item
The last two conditions in \eqref{D5'-3} in Definition~\ref{D5'}(i) are equivalent, respectively, to the following two:
\begin{gather*}
\frac{x_k-a_k}{\norm{x_k-a_k}}\norm{v_{1k}}\to v_1,
\quad
\frac{x_k-b_k}{\norm{x_k-b_k}}\norm{v_{2k}}\to v_2.
\end{gather*}
\item
Conditions~\eqref{D5'-3} in Definition~\ref{D5'}(i) can be replaced by conditions \eqref{D5'-1} and \eqref{D5'-2}.
\end{enumerate}
\end{proposition}

\begin{proof}
(i) The equivalence is a consequence of Lemma~\ref{L2} employed either with $u_k=\frac{x_k-a_k}{\norm{x_k-a_k}}$ and $v_k=v_{1k}$, or with $u_k=\frac{x_k-b_k}{\norm{x_k-b_k}}$ and $v_k=v_{2k}$.

(ii) Let sequences $(a_k)$, $(b_k)$, $(x_k)$, $(v_{1k})$ and $(v_{2k})$ satisfy all the conditions in Definition~\ref{D5'}(i).
Set
$$v_{1k}':=\frac{x_k-a_k}{\|x_k-a_k\|}\|v_{1k}\|,\quad v_{2k}':=\frac{x_k-b_k}{\|x_k-b_k\|}\|v_{2k}\|.$$
Then
$$\ang{v_{1k}',x_k-a_k}=\norm{v_{1k}'}\norm{x_k-a_k}, \quad \ang{v_{2k}',x_k-b_k}=\norm{v_{2k}'}\norm{x_k-b_k}.$$
Thanks to Lemma~\ref{L2}, it follows from the last two conditions in \eqref{D5'-3} that $v_{1k}'-v_{1k}\to0$ and $v_{2k}'-v_{2k}\to0$ as $k\to\infty$; hence, $v_{1k}'\to v_1$, $v_{2k}'\to v_2$, and conditions \eqref{D5'-1} and \eqref{D5'-2} are satisfied with $v_{1k}'$ and $v_{2k}'$ in place of $v_{1k}$ and $v_{2k}$, respectively.

Conversely, let sequences $(a_k)$, $(b_k)$, $(x_k)$, $(v_{1k})$ and $(v_{2k})$ satisfy conditions \eqref{D5'-1} and \eqref{D5'-2}.
Then, for any $k$, there exist $v_{1k}'\in N_A(a_k)$ and $v_{2k}'\in N_B(b_k)$ such that $$\norm{v_{1k}'-v_{1k}}<d(v_{1k},N_A(a_k))+\frac{1}{k}, \quad \norm{v_{2k}'-v_{2k}}<d(v_{2k},N_B(b_k))+\frac{1}{k},$$
and consequently, $v_{1k}'\to v_1$, $v_{2k}'\to v_2$, and (with the convention that $\frac{0}{0}=1$)
\begin{gather*}
\lim_{k\to\infty} \frac{\ang{v_{1k}',x_k-a_k}}{\norm{v_{1k}'}\norm{x_k-a_k}} =\lim_{k\to\infty} \frac{\ang{v_{1k},x_k-a_k}}{\norm{v_{1k}}\norm{x_k-a_k}}=1, \\
\lim_{k\to\infty} \frac{\ang{v_{2k}',x_k-b_k}}{\norm{v_{2k}'}\norm{x_k-b_k}} =\lim_{k\to\infty} \frac{\ang{v_{2k},x_k-b_k}}{\norm{v_{2k}}\norm{x_k-b_k}}=1;
\end{gather*}
hence, the conditions in Definition~\ref{D5'}(i) are satisfied with $v_{1k}'$ and $v_{2k}'$ in place of $v_{1k}$ and $v_{2k}$, respectively.
\qed\end{proof}

\begin{remark}
1. Thanks to Remark~\ref{D3}.1 and Proposition~\ref{P15}(i), the last two conditions in \eqref{D5'-3} in Definition~\ref{D5'}(i) can be replaced, respectively, by conditions
\begin{gather*}
\frac{x_k-a_k}{\norm{x_k-a_k}}\norm{v_{1}}\to v_1,
\quad
\frac{x_k-b_k}{\norm{x_k-b_k}}\norm{v_{2}}\to v_2.
\end{gather*}
When $v_1\ne0$ ($v_2\ne0$), one can write
\begin{gather*}
\frac{x_k-a_k}{\norm{x_k-a_k}}\to \frac{v_1}{\norm{v_1}}
\quad
\left(\frac{x_k-b_k}{\norm{x_k-b_k}}\to \frac{v_2}{\norm{v_2}}\right).
\end{gather*}

2. A replacement similar to the one in Proposition~\ref{P15}(ii) is possible for the relative limiting normals discussed in Remark~\ref{D3}.2: $v\in\overline{N}_{A}(\bar x;(x_k))$ if and only if there exist sequences $(a_k)\subset A$ and $(v_k)\subset X$ such that $a_k\ne x_k$ $(k=1,2,\ldots)$, $a_k\to\bx$, $v_k\to v$ and
\begin{gather*}
d(v_k,N_A(a_k))\to0,\quad
\ang{v_k,x_k-a_k}=\norm{v_k}\norm{x_k-a_k} \quad (k=1,2,\ldots).
\end{gather*}
\xqed\end{remark}

Next we show that in the Euclidian space setting the cone $\overline{N}_{A,B}(\bar x)$ of pairs of relative limiting normals can be replaced, when checking intrinsic transversality in accordance with Theorem~\ref{T3}, by the cone $\overline{N}{}^c_{A,B}(\bar x)$ of pairs of restricted relative limiting normals.

\begin{proposition}\label{Sets_equal}
Suppose $X$ is a Euclidean space, $A,B\subset X$, $\bx\in A\cap B$ and $v\ne0$.
Then $(v,-v)\in\overline{N}_{A,B}(\bar x)$ if and only if $(v,-v)\in\overline{N}{}^c_{A,B}(\bar x)$.
\end{proposition}
\begin{proof}
The `if' part follows immediately from the first inclusion in Proposition~\ref{P7++}(ii).
Conversely, let $(v,-v)\in \overline{N}_{A,B}(\bar x)$ and $(a_k)$, $(b_k)$, $(x_k)$, $(v_{1k})$ and $(v_{2k})$ be the corresponding sequences as in Definition~\ref{D5}(i) and such that $v_{1k}\to v$ and $v_{2k}\to -v$.
Using Lemma~\ref{L2}, it is not difficult to check that
\begin{align*}
\frac{x_k-a_k}{\|x_k-a_k\|}\to \frac{v}{\|v\|},\quad \frac{x_k-b_k}{\|x_k-b_k\|}\to -\frac{v}{\|v\|}.
\end{align*}
Passing to subsequences if necessary, we can assume that either $\|x_k-a_k\|\ge \|x_k-b_k\|$ or $\|x_k-a_k\|\le \|x_k-b_k\|$ for all $k=1,2,\ldots$.
Without loss of generality, it is sufficient to consider the first case only.
For any $k$, choose a $t_k\in]0,1]$ such that the point $x_k':=a_k+t_k(x_k-a_k)$ satisfies $\|x_k'-a_k\|=\|x_k'-b_k\|$.
This is always possible thanks to the continuity of the norm.
Then $x_k'\to\bx$, and we have
\begin{align}\label{Sets_equal-1}
\frac{x_{k}'-a_{k}}{\|x_{k}'-a_{k}\|} = \frac{x_{k}-a_{k}}{\|x_{k}-a_{k}\|}\to\frac{v}{\|v\|},
\end{align}
and either $x_k'=x_k$ or
\begin{align*}
\frac{x_{k}'-x_{k}}{\|x_{k}'-x_{k}\|} = \frac{a_{k}-x_{k}}{\|a_{k}-x_{k}\|}\to-\frac{v}{\|v\|},
\end{align*}
and consequently, by Lemma \ref{Tech},
\begin{align}\label{Sets_equal-2}
\frac{x_{k}'-b_{k}}{\|x_{k}'-b_{k}\|} =\frac{x_{k}'-x_k+x_k-b_{k}}{\|x_{k}'-x_k+x_k-b_{k}\|} \to-\frac{v}{\|v\|}.
\end{align}
It follows from \eqref{Sets_equal-1} and \eqref{Sets_equal-2} that
\begin{align*}
\frac{\ang{v_{1k},x_k'-a_k}}{\|v_{1k}\|\|x_k'-a_k\|} \to 1,\quad
\frac{\ang{-v_{2k},x_k'-b_k}}{\|v_{2k}\|\|x_k'-b_k\|} \to 1.
\end{align*}
Thus, sequences $(a_k)$, $(b_k)$, $(x_k')$, $(v_{1k})$ and $(v_{2k})$ satisfy all the conditions in Definition~\ref{D5'}(ii).
Hence, $(v,-v)\in \overline{N}{}^c_{A,B}(\bar x)$.
\qed
\end{proof}

Since $\overline{N}{}^c_{A,B}(\bar x)\subset \overline{N}_{A,B}(\bar x)$ (Proposition~\ref{P7++}(ii)), the next corollary strengthens Theorem~\ref{T3} (in the Euclidean space setting).

\begin{corollary}\label{C4}
Suppose $X$ is a Euclidean space, $A,B\subset X$ are closed, and $\bx\in A\cap B$.
The following conditions are equivalent:
\begin{enumerate}
\item
$\{A,B\}$ is intrinsically transversal at $\bar x$;
\item
there exists a number $\alpha\in]0,1[$ such that ${\|v_1+v_2\|>\alpha}$
for all $(v_1,v_2)\in\overline{N}{}^c_{A,B}(\bar x)$
satisfying
$\|v_1\|+\|v_2\|=1$;
\item
$\left\{v\in X\mid (v,-v)\in\overline{N}{}^c_{A,B}(\bar x)\right\} \subset \{0\}$.
\end{enumerate}
Moreover, the exact upper bound of all $\alpha$ in {\rm (ii)} equals $\strc$.
\end{corollary}

Comparing Corollary~\ref{C4} with Theorem~\ref{T3+} and taking into account Proposition~\ref{P10}(iii), we arrive at the following equivalences of the three transversality properties for closed convex sets in Euclidian spaces.

\begin{corollary}\label{C4+}
Suppose $X$ is a Euclidean space, $A,B\subset X$ are closed and convex, and $\bx\in A\cap B$.
The following conditions are equivalent:
\begin{enumerate}
\item
$\{A,B\}$ is intrinsically transversal at $\bar x$;
\item
$\{A,B\}$ is weakly intrinsically transversal at $\bar x$;
\item
$\{A,B\}$ is subtransversal at $\bar x$.
\end{enumerate}
\end{corollary}

\begin{remark}
Transversality is in general stronger than all the properties above, even in the convex setting.
\end{remark}

It is well known that in Euclidean spaces Fr\'echet normals and subdifferentials can be approximated by proximal ones; see e.g. \cite[Exercise~6.18 and Corollary~8.47]{RocWet98}.
As a result, in many statements proximal normals can replace Fr\'echet ones.
This is true, in particular, when characterising intrinsic transversality.
The next proposition is a proximal version of Definition~\ref{D5'}(i).

\begin{proposition}\label{P13}
Suppose $X$ is a Euclidean space,
$A,B\subset X$ are closed, and $\bx\in A\cap B$.
Then $(v_1,v_2)\in\overline{N}_{A,B}(\bar x)$ if and only if there exist sequences $(a_k)\subset A\setminus B$, $(b_k)\subset B\setminus A$, $(x_k),(v_{1k}),(v_{2k})\subset X$ such that
$x_k\ne a_k$, $x_k\ne b_k$ $(k=1,2,\ldots)$, $a_k\to\bx$, $b_k\to\bx$, $x_k\to\bx$, $\frac{\norm{x_k-a_k}}{\norm{x_k-b_k}} \to1$, $v_{1k}\to v_1$, $v_{2k}\to v_2$ as $k\to\infty$, and
\begin{gather}\label{P13-1}
v_{1k}\in N_A^p(a_k),\;\; v_{2k}\in N_B^p(b_k)\;\; (k=1,2,\ldots),
\;\;
\frac{\ang{v_{1k},x_k-a_k}}{\norm{v_{1k}}\norm{x_k-a_k}} \to1,\;\; \frac{\ang{v_{2k},x_k-b_k}}{\norm{v_{2k}}\norm{x_k-b_k}} \to1,
\end{gather}
with the convention that $\frac{0}{0}=1$.
\end{proposition}

\begin{proof}
Since the proximal normal cone is always a subset of the Fr\'echet normal cone,
the `if' part is trivial.

Conversely, let $(v_1,v_2)\in\overline{N}_{A,B}(\bar x)$.
If $v_1=0$, for any $a_k\to\bx$, one can take $v_{1k}=0\in N_A^p(a_k)$ $(k=1,2\ldots)$.
Let $v_1\ne0$ and sequences $(x_k)$, $(a_k)$ and $(v_{1k})$ with $v_{1k}\ne0$ satisfy the conditions in Definition~\ref{D5}(i) with $v_1$ and $v_{1k}$ in place of $x_1^*$ and $x_{1k}^*$, respectively.
For each $k=1,2\ldots$, since $v_{1k}\in N_A(a_k)$, there exists a $\de>0$ such that
\begin{align}\label{P13P-1}
\ang{v_{1k},a-a_k}\le\frac{1}{4k}\norm{v_{1k}}\norm{a-a_k}
\qdtx{for all} a\in A\cap\B_\de(a_k).
\end{align}
Take a $t_k>0$ such that
\begin{align}\label{P13P-3}
t_k< \min\left\{\frac{\de}{2},\norm{x_{k}-a_k}, \frac{1}{2}d(a_k,B)\right\} \norm{v_{1k}}\iv,
\end{align}
set $x_{k}':=a_k+t_kv_{1k}$, and choose an $a_k'\in P_A(x_{k}')$.
Then
\begin{gather}\notag
\norm{x_{k}'-a_k}=t_k\norm{v_{1k}}<\norm{x_{k}-a_k},
\\\label{P13P-2}
\norm{a_k'-a_k}^2
=\norm{x_{k}'-a_k'}^2-\norm{x_{k}'-a_k}^2+ 2\ang{x_{k}'-a_k,a_k'-a_k}
\le 2t_k\ang{v_{1k},a_k'-a_k},
\\\label{P13P-4}
\norm{a_k'-a_k}\le2t_k\norm{v_{1k}}<\de,
\\\notag
d(a_k',B)\ge d(a_k,B)-\norm{a_{k}'-a_k}\ge d(a_k,B)-2t_k\norm{v_{1k}}>0.
\end{gather}
We have
$a_k'\in A\setminus B$, $x_{k}'\to\bx$, $a_k'\to\bx$ as $k\to\infty$.
It follows from \eqref{P13P-2}, \eqref{P13P-4} and \eqref{P13P-1} that
\begin{align*}
\norm{a_k'-a_k}\le\frac{t_k}{2k}\norm{v_{1k}} =\frac{1}{2k}\norm{x_{k}'-a_k},
\end{align*}
and consequently,
\begin{align*}
\norm{x_{k}'-a_k'}&\le\norm{x_{k}'-a_k},
\\
\norm{x_{k}'-a_k'}&\ge\norm{x_{k}'-a_k}-\norm{a_k'-a_k}\ge \left(1-\frac{1}{2k}\right)\norm{x_{k}'-a_k}>0,
\\
\norm{x_{k}-a_k'}&\le\norm{x_{k}-a_k}+\norm{a_k'-a_k}< \left(1+\frac{1}{2k}\right)\norm{x_{k}-a_k},
\\
\norm{x_{k}-a_k'}&\ge\norm{x_{k}-a_k}-\norm{a_k'-a_k}> \left(1-\frac{1}{2k}\right)\norm{x_{k}-a_k}.
\end{align*}
Hence,
\begin{align*}
\lim_{k\to\infty} \frac{\norm{a_{k}'-a_k}}{\norm{x_{k}'-a_k}}=0,
\quad
\lim_{k\to\infty} \frac{\norm{x_{k}'-a_k'}}{\norm{x_{k}'-a_k}}
=\lim_{k\to\infty} \frac{\norm{x_{k}-a_k'}}{\norm{x_{k}-a_k}}=1.
\end{align*}
Set
\begin{align*}
v_{1k}':=\frac{x_{k}'-a_k'}{\norm{x_{k}'-a_k'}}\norm{v_{1k}}.
\end{align*}
We obviously have $v_{1k}'\in N_A^p(a_k')\setminus\{0\}$, $\norm{v_{1k}'}=\norm{v_{1k}}$,
\begin{align*}
\lim_{k\to\infty}v_{1k}' =\lim_{k\to\infty} \frac{x_{k}'-a_k'}{\norm{x_{k}'-a_k'}}\norm{v_{1k}} =\lim_{k\to\infty} \frac{x_{k}'-a_k}{\norm{x_{k}'-a_k}}\norm{v_{1k}}
=\lim_{k\to\infty} v_{1k}=v,
\end{align*}
\begin{align*}
\lim_{k\to\infty} \frac{\ang{v_{1k}',x_k-a_k'}}{\norm{v_{1k}'}\norm{x_k-a_k'}} &=\lim_{k\to\infty} \frac{\ang{x_{k}'-a_k',x_k-a_k'}} {\norm{x_{k}'-a_k'}\norm{x_k-a_k'}}
\\
&=\lim_{k\to\infty} \frac{\ang{x_{k}'-a_k,x_k-a_k}} {\norm{x_{k}'-a_k}\norm{x_k-a_k}}
\\
&=\lim_{k\to\infty}\frac{\ang{v_{1k},x_k-a_k}} {\norm{v_{1k}}\norm{x_k-a_k}}=1.
\end{align*}
Thus the sequences $(a_k')$, and $(v_{1k}')$ satisfy the conditions in the proposition.
Similarly, given a $v_2$ and sequences $(x_k)$, $(b_k)$ and $(v_{2k})$ satisfying the conditions in Definition~\ref{D5}(i), one can construct sequences $(b_k')$, and $(v_{2k}')$ satisfying the conditions in the proposition.
This concludes the proof.
\qed\end{proof}

\begin{remark}
1. In Proposition~\ref{P13}, one can always assume that $\norm{v_{1k}}=\norm{v_1}$, $\norm{v_{2k}}=\norm{v_2}$, $(k=1,2,\ldots)$; cf. Remark~\ref{D3}.1.

2. For the set $\overline{N}_{A}(\bar x;(x_k))$ of limiting normals to $A$ at $\bx$ relative to $(x_k)$ defined in Remark~\ref{D3}.2, similarly to  Proposition~\ref{P13}, one can show that $v\in\overline{N}_{A}(\bar x;(x_k))$ if and only if there exist sequences $(a_k)\subset A$ and $(v_k)\subset X$ such that
\begin{gather*}\notag
a_k\ne x_k,\quad v_{k}\in N_A^p(a_k)\quad (k=1,2,\ldots),
\quad
a_k\to\bx,\quad v_{k}\to v,
\quad
\frac{\ang{v_{k},x_k-a_k}}{\norm{v_k}\norm{x_k-a_k}} \to1,
\end{gather*}
with the convention that $\frac{0}{0}=1$.
\xqed\end{remark}

The next statement is a proximal version of Definition~\ref{D3+}(ii).
It is a consequence of Theorem~\ref{T3} and Proposition~\ref{P13}.

\begin{theorem}\label{T4}
Suppose $X$ is a Euclidean space,
$A,B\subset X$ are closed, and $\bx\in A\cap B$.
Then $\{A,B\}$ is intrinsically transversal
at $\bx$ if and only if there exist numbers $\alpha\in]0,1[$ and $\rho>0$ such that ${\|v_{1}+v_{2}\|>\alpha}$
for all $a\in(A\setminus B)\cap\B_\de(\bx)$, $b\in(B\setminus A)\cap\B_\de(\bx)$, $x\in\B_\de(\bx)$ with
$x\ne a$, $x\ne b$, $1-\delta<\frac{\norm{x-a}}{\norm{x-b}}<1+\delta$,
and all $v_{1}\in N_{A}^p(a)$, $v_{2}\in N_{B}^p(b)$
satisfying
\begin{gather*}\notag
\|v_{1}\|+\|v_{2}\|=1,\quad
\frac{\ang{v_1,x-a}} {\|v_1\|\|x-a\|} >1-\rho,\quad
\frac{\ang{v_2,x-b}} {\|v_2\|\|x-b\|} >1-\rho,
\end{gather*}
with the convention that $\frac{0}{0}=1$.
Moreover, the exact upper bound of all such $\al$ equals $\itr$.
\end{theorem}

\section{More characterizations of intrinsic transversality in Euclidian spaces}\label{S4}

In Euclidian spaces one can go further than restricting the set of relative limiting normals when computing the dual space intrinsic transversality constant \eqref{itr+} to only nonzero ones as observed in Remark~\ref{R6}.2: it is sufficient to consider only unit normals.

\begin{proposition}\label{P17}
Suppose $X$ is a Euclidean space, $A,B\subset X$ are closed, and $\bx\in A\cap B$.
Then
\begin{gather}\label{P17-1}
\itr= \frac{1}{2} \min_{\substack{(v_1,v_2)\in\overline{N}_{A,B}(\bar x)\\\|v_1\|=\|v_2\|=1}}\|v_1+v_2\|,
\end{gather}
with the convention that the minimum over the empty set equals $2$.
\end{proposition}

\begin{proof}
If there is no pair $(v_1,v_2)\in\overline{N}_{A,B}(\bar x)$ with $v_1\ne0$ and $v_2\ne0$, then both sides in \eqref{P17-1} equal 1.

Given any $(v_1,v_2)\in\overline{N}_{A,B}(\bar x)$ with $\|v_1\|=\|v_2\|=1$, we have $(\frac{v_1}{2},\frac{v_2}{2})\in\overline{N}_{A,B}(\bar x)$ and $\norm{\frac{v_1}{2}}+\norm{\frac{v_2}{2}}=1$.
Hence, by \eqref{itr+},
\begin{gather*}
\itr\le \frac{1}{2}\norm{v_1+v_2},
\end{gather*}
and consequently,
\begin{gather*}
\itr\le \frac{1}{2} \min_{\substack{(v_1,v_2)\in\overline{N}_{A,B}(\bar x)\\\|v_1\|=\|v_2\|=1}}\|v_1+v_2\|.
\end{gather*}

On the other hand, notice that, when evaluating the minimum in \eqref{itr+}, it is sufficient to consider only $(v_1,v_2)\in\overline{N}_{A,B}(\bar x)$ with $v_1\ne0$ and $v_2\ne0$.
Given any $(v_1,v_2)\in\overline{N}_{A,B}(\bar x)$ with $v_1\ne0$, $v_2\ne0$ and $\|v_1\|+\|v_2\|=1$, we define $w_1=\frac{v_1}{\norm{v_1}}$ and $w_2=\frac{v_2}{\norm{v_2}}$.
Then $(w_1,w_2)\in\overline{N}_{A,B}(\bar x)$, $\norm{w_1}=\norm{w_2}=1$, and by Lemma~\ref{L3},
\begin{gather*}
\norm{v_1+v_2}\ge \frac{1}{2}\norm{w_1+w_2}. \end{gather*}
Hence,
\begin{gather*}
\frac{1}{2} \min_{\substack{(w_1,w_2)\in\overline{N}_{A,B}(\bar x)\\\|w_1\|=\|w_2\|=1}}\|w_1+w_2\|\le\norm{v_1+v_2},
\end{gather*}
and consequently, by \eqref{itr+},
\begin{gather*}
\frac{1}{2} \min_{\substack{(w_1,w_2)\in\overline{N}_{A,B}(\bar x)\\\|w_1\|=\|w_2\|=1}}\|w_1+w_2\|\le\itr.
\end{gather*}
The proof is completed.
\qed\end{proof}

Alongside $\itr$, several other
constants can be used for characterizing intrinsic transversality in Euclidean spaces:
\begin{align}\label{std1}
\itrd{1}:=& \max_{\substack{(v_1,v_2)\in\overline{N}_{A,B}(\bar x)\\\|v_1\|=\|v_2\|=1}}\|v_1-v_2\|,
\\\label{std2}
\itrd{2}:=& -\min_{\substack{(v_1,v_2)\in\overline{N}_{A,B}(\bar x)\\\|v_1\|=\|v_2\|=1}}\ang{v_1,v_2},
\\\label{std3}
\itrd{3}:=& \min_{\norm{v}=1}d((v,-v),\overline{N}_{A,B}(\bar x)),
\end{align}
with the Euclidean distance in $X\times X$ used in \eqref{std3} and the conventions that in \eqref{std1} and \eqref{std2} the maximum and minimum over the empty set equal $0$ and $1$, respectively, and the distance to the empty set in \eqref{std3} equals the distance to the origin, i.e., $\sqrt{2}$.

The expression $\ang{v_1,v_2}$ in \eqref{std2} can be interpreted as the cosine of the angle between the vectors $v_1$ and $v_2$.
Taking the minimum means minimizing the cosine or, equivalently, maximizing the angle, forcing the vectors to go in opposite directions, potentially making the angle obtuse (or even equal $-\pi$).
In this case, $\itrd{2}>0$.
However, in general, unlike $\itr$, $\itrd{1}$ and $\itrd{3}$, constant $\itrd{2}$ can be negative.

The relationships between each of the constants \eqref{std1}, \eqref{std2} and \eqref{std3} and the original dual space constant $\itr$ \eqref{itrw} (cf. its equivalent representations in \eqref{itr+} and \eqref{P17-1}) are given in the next proposition.
They follow from the geometry of Euclidean space.

\begin{proposition}\label{P18}
Suppose $X$ is a Euclidean space, $A,B\subset X$ are closed, and $\bx\in A\cap B$.
Then
\begin{gather}\label{P18-1}
(\itr)^2+\frac{1}{4} (\itrd{1})^2=1,
\\\label{P18-2}
\itrd{2}+2(\itr)^2=1.
\end{gather}
If either
$\overline{N}_{A,B}(\bar x)$ contains a pair $(v_1,v_2)\in X\times X$ of nonzero positively independent vectors (none of the vectors is a positive multiple of the other),
or $\overline{N}_{A,B}(\bar x)=\{(0,0)\}$ or $\overline{N}_{A,B}(\bar x)=\emptyset$, then
\begin{gather}\label{P18-3}
\itrd{3}=\sqrt{2}\,\itr;
\end{gather}
otherwise $\itrd{3}=\itr=1$.
\end{proposition}

\begin{proof}
If there is no pair $(v_1,v_2)\in\overline{N}_{A,B}(\bar x)$ with $v_1\ne0$ and $v_2\ne0$, then, in accordance with the conventions made, $\itr=1$ (see \eqref{P17-1}), $\itrd{1}=0$, $\itrd{2}=-1$; so equalities \eqref{P18-1} and \eqref{P18-2} hold true.
If either $\overline{N}_{A,B}(\bar x)=\{(0,0)\}$ or $\overline{N}_{A,B}(\bar x)=\emptyset$, then $\itrd{3}=\sqrt{2}$ and equality \eqref{P18-3} holds true.

For any $v_1,v_2\in X$, one has
\begin{gather}\label{P18P-0}
\norm{v_1+v_2}^2 =\norm{v_1}^2+\norm{v_2}^2+2\ang{v_1,v_2},
\\\label{P18P-1}
\norm{v_1+v_2}^2+\norm{v_1-v_2}^2 =2\left(\norm{v_1}^2+\norm{v_2}^2\right).
\end{gather}
From \eqref{P17-1}, \eqref{P18P-1} and \eqref{std1}, we obtain
\begin{gather*}
(\itr)^2=1-\frac{1}{4} \max_{\substack{(v_1,v_2)\in\overline{N}_{A,B}(\bar x)\\\|v_1\|=\|v_2\|=1}}\norm{v_1-v_2}^2 =1-\frac{1}{4}(\itrd{1})^2,
\end{gather*}
which proves \eqref{P18-1}.
Similarly, from \eqref{P17-1}, \eqref{P18P-0} and \eqref{std2},
\begin{gather*}
2(\itr)^2=1+ \min_{\substack{(v_1,v_2)\in\overline{N}_{A,B}(\bar x)\\\|v_1\|=\|v_2\|=1}}\ang{v_1,v_2}=1-\itrd{2},
\end{gather*}
which proves \eqref{P18-2}.

Definition \eqref{std3} can be rewritten as follows:
\begin{align}\label{P18P-2}
(\itrd{3})^2= \min_{\|v\|=1,\,(v_1,v_2)\in\overline{N}_{A,B}(\bar x)} \left(\norm{v-v_1}^2+\norm{v+v_2}^2\right)
\end{align}
with the convention that the minimum over the empty set equals $\sqrt{2}$.
We next prove equality \eqref{P18-3} in the nontrivial case when
$\overline{N}_{A,B}(\bar x)$ contains a pair $(v_1,v_2)$ of nonzero vectors with none of them being a positive multiple of the other.
Let the minimum in \eqref{P18P-2} be attained at some $v\in X$ with $\norm{v}=1$ and $(v_1,v_2)\in \overline{N}_{A,B}(\bar x)$.
Then $v_1$ and $-v_2$ are the projections of $v$ on the rays $R_1$ and $R_2$ determined by $v_1$ and $-v_2$, respectively.
In general, one of the rays or both can be trivial.
However, in the the nontrivial case, we can restrict ourselves to the pairs $(v_1,v_2)$ described above.
Thus, $v_1\ne v_2$, the ray $R_1$ and $R_2$ are nontrivial and do not go in opposite directions.
It also follows from \eqref{P18P-2} that $v$ must lie in the plane determined by the ray $R_1$ and $R_2$ in such a way that $\ang{v,v_1}\ge0$ and $\ang{v,v_2}\le0$.
Since $v_1$ and $-v_2$ are the projections of $v$ on the rays $R_1$ and $R_2$, we have
\begin{align}\label{P18P-3}
\ang{v,v_1}=\|v_1\|^2,
\quad
-\ang{v,v_2}=\|v_2\|^2,
\end{align}
and with $\|v\|=1$ the expression under the $\min$ in \eqref{P18P-2} takes the following form:
\begin{align}\label{P18P-4}
\norm{v-v_1}^2+\norm{v+v_2}^2 =2-\ang{v,v_1}+\ang{v,v_2}=2-\ang{v,v_1-v_2}.
\end{align}
Since $v$ minimizes this expression over the unit sphere, we have
\begin{align}\label{P18P-5}
v=\frac{v_1-v_2}{\norm{v_1-v_2}}.
\end{align}
Hence, in view of \eqref{P18P-3},
\begin{gather*}
\norm{v_1-v_2}=\ang{v,v_1-v_2}=\norm{v_1}^2+\norm{v_2}^2,
\\
\|v_1+v_2\|^2=
2\left(\|v_1\|^2+\|v_2\|^2\right)-\|v_1-v_2\|^2= \left(\|v_1\|^2+\|v_2\|^2\right) \left(2-\norm{v_1-v_2}\right).
\end{gather*}
It follows now from \eqref{P18P-4} that
\begin{align}\label{P18P-6}
\norm{v-v_1}^2+\norm{v+v_2}^2 =2-\norm{v_1-v_2}=\frac{\|v_1+v_2\|^2}{\|v_1\|^2+\|v_2\|^2}.
\end{align}
In view of \eqref{P18P-3} and \eqref{P18P-5}, we also have
\begin{align*}
\norm{v_1}^2(1-\norm{v_1-v_2})=\ang{v_1,v_2}, \quad \norm{v_2}^2(1-\norm{v_1-v_2})=\ang{v_1,v_2}.
\end{align*}
The last two equalities imply that either $\norm{v_1}=\norm{v_2}$, or $\norm{v_1-v_2}=1$ and $\ang{v_1,v_2}=0$, i.e., the rays $R_1$ and $R_2$ are orthogonal.
In the last case, any pair $v_1\in R_1$, $v_2\in R_2$, with $\norm{v_1-v_2}=1$ minimizes expression \eqref{P18P-6} (the minimum equals 1), and we choose $v_1$ and $v_2$ such that $\norm{v_1}=\norm{v_2}=\frac{1}{\sqrt{2}}$.
Thus, in both cases $\norm{v_1}=\norm{v_2}$, and it follows from \eqref{P18P-6} that
\begin{align}\label{P18P-7}
\norm{v-v_1}^2+\norm{v+v_2}^2 =\frac{1}{2}\|v_1'+v_2'\|^2,
\end{align}
where $v_1':=\frac{v_1}{\norm{v_1}}$ and $v_2':=\frac{v_2}{\norm{v_2}}$.
Obviously, $(v_1',v_2')\in\overline{N}_{A,B}(\bar x)$, $\|v_1'\|=\|v_2'\|=1$, and it follows from \eqref{P18P-2}, \eqref{P18P-7} and \eqref{P17-1} that
\begin{align}\label{P18P-8}
(\itrd{3})^2&=\|v-v_1\|^2+\|v+v_2\|^2
\ge2(\itr)^2.
\end{align}
Conversely, let the minimum in \eqref{P17-1}
be attained at some $(v_1',v_2')\in\overline{N}_{A,B}(\bar x)$ with $\|v_1'\|=\|v_2'\|=1$.
Choose a unit vector $v$ such that $\ang{v,v_1'-v_2'}=\|v_1'-v_2'\|$, and let $v_1$ and $-v_2$ be the projections of $v$ on  the rays determined by $v_1'$ and $-v_2'$, respectively.
We are in a situation as above and, using \eqref{P18P-7} again, we obtain
\begin{align}\label{P18P-9}
4(\itr)^2&=\|v_1'+v_2'\|^2= 2(\|v-v_1\|^2+\|v+v_2\|^2)\ge2(\itrd{3})^2.
\end{align}
Combining \eqref{P18P-8} and \eqref{P18P-9} proves \eqref{P18-3}.

Now we consider the case when one of the components of $\overline{N}_{A,B}(\bar x)$ is trivial while the other one is not.
Let, e.g., $\overline{N}_{A,B}(\bar x)=C\times\{0\}$, where $C$ is a nontrivial ($C\ne\emptyset$ and $C\ne\{0\}$) cone in $X$.
Then by \eqref{P18P-2}, $\itrd{3}\ge1$, and, given any $\hat v\in C$ with $\|\hat v\|=1$, one can take $v=\hat v$ to get $\itrd{3}\le\|\hat v\|=1$.
Hence, $\itrd{3}=1$.
In this case $\itr=1$ by convention.

Finally we consider the remaining case when $\overline{N}_{A,B}(\bar x)=R\times R$ where $R$ is a ray in $X$ determined by a unit vector $\bar v$.
By definition \eqref{std3},
\begin{align*}
(\itrd{3})^2= \min_{\|v\|=1} \left(d^2(v,R)+d^2(-v,R)\right).
\end{align*}
For any $v$, one of the distances in the above expression is attained at the origin and equals 1.
Hence, $\itrd{3}\ge1$.
On the other hand, with $v=\hat v$ we have $\itrd{3}\le d(-\hat v,R)=1$, and consequently, $\itrd{3}=1$.
In this case, by \eqref{P17-1}, $\itr=\frac{1}{2}\norm{\hat v+\hat v}=1$.
\qed\end{proof}

\begin{remark}
The only property of $\overline{N}_{A,B}(\bar x)$ used in the proof of Proposition~\ref{P18} is
the one in Proposition~\ref{P7++}(i).
The proof is applicable in other situations, e.g., when establishing similar relationships between the dual space constants characterizing the transversality property introduced in Definition~\ref{D1}(ii) (cf. \cite{KruTha13,KruLukTha}).
One only needs to replace $\overline{N}_{A,B}(\bar x)$ in the above proof with $\overline{N}_{A}(\bar x)\times\overline{N}_{B}(\bar x)$ where $\overline{N}_{A}(\bar x)$ and $\overline{N}_{B}(\bar x)$ are conventional limiting normal cones (cf. definition \eqref{NC3}) at $\bx$ to the sets $A$ and $B$, respectively.
\xqed\end{remark}
Thanks to Propositions~\ref{P17} and \ref{P18}, the limiting criteria of intrinsic transversality in Theorem~\ref{T3} can be complemented in the Euclidean space setting by several more criteria collected in the next theorem.

\begin{theorem}\label{T5}
Suppose $X$ is a Euclidean space,
$A,B\subset X$ are closed, and $\bx\in A\cap B$.
The following conditions are equivalent:
\begin{enumerate}
\item
$\{A,B\}$ is intrinsically transversal at $\bar x$;
\item
there exists a number $\alpha\in]0,1[$ such that ${\|v_{1}+v_{2}\|>\alpha}$ for all $(v_1,v_2)\in\overline{N}_{A,B}(\bar x)$ with $\|v_1\|={\|v_2\|=1}$;
the exact upper bound of all such $\al$ equals $2\,\itr$;
\item
$\itrd{1}<2$,\\
i.e., there exists a number $\alpha<2$ such that ${\|v_{1}-v_{2}\|<\alpha}$ for all $(v_1,v_2)\in\overline{N}_{A,B}(\bar x)$ with $\|v_1\|=\|v_2\|=1$;
the exact lower bound of all such $\al$ equals $\itrd{1}$;
\item
$\itrd{2}<1$,\\
i.e., there exists a number $\alpha<1$ such that $\ang{v_{1},v_{2}}>-\alpha$ for all $(v_1,v_2)\in\overline{N}_{A,B}(\bar x)$ with $\|v_1\|=\|v_2\|=1$;
the exact lower bound of all such $\al$ equals $\itrd{2}$;
\item
$\itrd{3}>0$,\\
i.e., there exists a number $\alpha\in]0,1[$ such that $d\left((v,-v),\overline{N}_{A,B}(\bar x)\right)> \alpha$ for all $v\in X$ with $\|v\|=1$;
the exact upper bound of all such $\al$ equals $\itrd{3}$.
\end{enumerate}
\end{theorem}

\begin{remark}\label{R11}
1. Instead of the constant $\itrd{2}$, defined by \eqref{std2}, one can employ in Theorem~\ref{T5}(iv) its modification:
\begin{align*}
\itrdd{2}:=& -\min_{\substack{(v_1,v_2)\in\overline{N}_{A,B}(\bar x)\\\|v_1\|\le1,\,\|v_2\|\le1}}\ang{v_1,v_2}.
\end{align*}
It is easy to check that $\itrdd{2}=(\itrd{2})_+$.
Hence, the last constant is always nonnegative, and $\itrd{2}<1$ if and only if $\itrdd{2}<1$.

2. In view of Definition~\ref{D5}(i) and Proposition~\ref{P13}, the intrinsic transversality constants admit equivalent sequential representations in terms of Fr\'echet (or proximal) normals to $A$ and $B$ computed at points near $\bx$.
For instance,
\begin{gather}\label{std+}
\itr= \frac{1}{2} \liminf_{\substack{
a\to\bx,\, b\to\bx,\, x\to\bx \\
a\in A\setminus B,\, b\in B\setminus A,\, a\ne x,\, b\ne x,\\
v_{1}\in N_A(a),\, v_{2}\in N_B(b),\, \|v_{1}\|=\|v_{2}\|=1\\
\frac{\norm{x-a}}{\norm{x-b}}\to1\; \frac{\ang{v_{1},x-a}}{\norm{x-a}} \to1,\, \frac{\ang{v_{2},x-b}}{\norm{x-b}} \to1}} \norm{v_{1}+v_{2}},
\\\label{std2+}
\itrd{2}= -\liminf_{\substack{
a\to\bx,\, b\to\bx,\, x\to\bx \\
a\in A\setminus B,\, b\in B\setminus A,\, a\ne x,\, b\ne x\\
v_{1}\in N_A(a),\, v_{2}\in N_B(b),\, \|v_{1}\|=\|v_{2}\|=1\\
\frac{\norm{x-a}}{\norm{x-b}}\to1,\; \frac{\ang{v_{1},x-a}}{\norm{x-a}} \to1,\, \frac{\ang{v_{2},x-b}}{\norm{x-b}} \to1}} \ang{v_{1},v_{2}},
\end{gather}
with the convention that the infimum over the empty set in \eqref{std+} and \eqref{std2+} equals 2 and 1, respectively.
Each of the criteria of  intrinsic transversality in Theorem~\ref{T5} can be rewritten equivalently in terms of Fr\'echet (or proximal) normals to $A$ and $B$ computed at points near $\bx$.

3. Thanks to Lemma~\ref{L2}, one can write down several more equivalent representations.
For instance,
\begin{align}\notag
\itr &= \frac{1}{2} \liminf_{\substack{
a\to\bx,\, b\to\bx,\, x\to\bx \\
a\in A\setminus B,\, b\in B\setminus A,\, a\ne x,\, b\ne x\\
v_{1}\in N_A(a),\, v_{2}\in N_B(b),\, \|v_{1}\|=\|v_{2}\|=1\\ \frac{\norm{x-a}}{\norm{x-b}}\to1,\; \frac{x-a}{\norm{x-a}}-v_{1}\to0,\, \frac{x-b}{\norm{x-b}}-v_{2}\to0}} \norm{v_{1}+v_{2}}
\\\notag
&= \frac{1}{2} \liminf_{\substack{
a\to\bx,\, b\to\bx,\, x\to\bx \\
a\in A\setminus B,\, b\in B\setminus A,\, a\ne x,\, b\ne x\\
v_{1}\in N_A(a),\, v_{2}\in N_B(b),\, \|v_{1}\|=\|v_{2}\|=1\\ \frac{\norm{x-a}}{\norm{x-b}}\to1,\; \frac{x-a}{\norm{x-a}}-v_{1}\to0,\, \frac{x-b}{\norm{x-b}}-v_{2}\to0}} \norm{\frac{x-a}{\norm{x-a}}+\frac{x-b}{\norm{x-b}}}
\\\label{std+2}
&= \frac{1}{2} \liminf_{\substack{
a\to\bx,\, b\to\bx,\, x\to\bx \\
a\in A\setminus B,\, b\in B\setminus A,\, a\ne x,\, b\ne x,\;\frac{\norm{x-a}}{\norm{x-b}}\to1\\
d\left(\frac{x-a}{\norm{x-a}},N_A(a)\right)\to0,\, d\left(\frac{x-b}{\norm{x-b}},N_B(b)\right)\to0}} \norm{\frac{x-a}{\norm{x-a}}+\frac{x-b}{\norm{x-b}}},
\\\label{std2+2}
\itrd{2}&= -\liminf_{\substack{
a\to\bx,\, b\to\bx,\, x\to\bx \\
a\in A\setminus B,\, b\in B\setminus A,\, a\ne x,\, b\ne x,\;\frac{\norm{x-a}}{\norm{x-b}}\to1\\
d\left(\frac{x-a}{\norm{x-a}},N_A(a)\right)\to0,\, d\left(\frac{x-b}{\norm{x-b}},N_B(b)\right)\to0}} \frac{\ang{x-a,x-b}} {\norm{x-a}\norm{x-b}},
\end{align}
with the convention that the infimum over the empty set in \eqref{std+2} and \eqref{std2+2} equals 2 and 1, respectively.
\xqed\end{remark}

Another pair of constants originated in \cite{DruIofLew15} can be of interest:
\begin{align}\notag
\itrdh{1}
:=&\liminf_{\substack{
a\to\bx,\, b\to\bx
\\\label{std4}
a\in A\setminus B,\, b\in B\setminus A
}}
\max\biggl\{d\left(\frac{b-a}{\|a-b\|}, N_{A}(a)\right),
d\left(\frac{a-b}{\|a-b\|}, N_{B}(b)\right)\biggr\}
\\
=&\liminf_{\substack{
a\to\bx,\, b\to\bx\\
a\in A\setminus B,\, b\in B\setminus A\\
v_1\in N_{A}(a),\,v_2\in N_{B}(b)
}}
\max\biggl\{\norm{\frac{b-a}{\|a-b\|}-v_1},
\norm{\frac{a-b}{\|a-b\|}-v_2}\biggr\},
\end{align}
\begin{gather}\label{std5}
\itrdh{2}
:=\limsup_{\substack{
a\to\bx,\, b\to\bx
\\
a\in A\setminus B,\, b\in B\setminus A
\\
v_1\in N_{A}(a),\,v_2\in N_{B}(b),\, \norm{v_1}=\norm{v_2}=1}}
\left[\min\biggl\{\ang{\frac{b-a}{\|a-b\|},v_1},
\ang{\frac{a-b}{\|a-b\|},v_2}\biggr\}\right]_+,
\end{gather}
with the convention that the infimum and supremum over the empty set equal 1 and 0, respectively.
Thanks to this convention, it always holds $0\le\itrdh{1}\le1$ and $0\le\itrdh{2}\le1$.

\begin{remark}\label{R14}
1. Points $a\in A\setminus B$ and $b\in B\setminus A$ with either $N_{A}(a)=\{0\}$ or $N_{B}(b)=\{0\}$ can be excluded from definition \eqref{std4} because at such points either $d\left(\frac{b-a}{\|a-b\|},N_{A}(a)\right)=1$ or $d\left(\frac{a-b}{\|a-b\|},N_{B}(b)\right)=1$.

2. Fr\'echet normal cones in representations \eqref{std+}, \eqref{std2+}, \eqref{std+2}, \eqref{std2+2}, \eqref{std4} and \eqref{std5} can be replaced by proximal or limiting ones.
\xqed\end{remark}

\begin{proposition}\label{P19}
Suppose $X$ is a Euclidean space, $A,B\subset X$ are closed, and $\bx\in A\cap B$. Then
\begin{enumerate}
\item
$(\itrdh{1})^2+(\itrdh{2})^2=1$;
\item
if $\itr<\frac{1}{\sqrt{2}}$, then
$\itrdh{1}\le2\itr\sqrt{1-(\itr)^2}$;
\item
$\itrdh{1}=0$ if and only if $\itr=0$.
\end{enumerate}
\end{proposition}

\begin{proof}
(i) If there are no points $a\in A\setminus B$ and $b\in B\setminus A$ in a neighbourhood of $\bx$ with $N_{A}^p(a)\ne\{0\}$ and $N_{B}^p(b)\ne\{0\}$, then $\itrdh{1}=1$ and $\itrdh{2}=0$ in view of the conventions made and Remark~\ref{R14}; hence, equality (i) holds.

Let $a\in A\setminus B$ and $b\in B\setminus A$, $v_1\in N_{A}^p(a)$, $v_2\in N_{B}^p(b)$, and $\norm{v_1}=\norm{v_2}=1$.
Denote
\begin{gather}\label{P19P-3}
u:=\frac{b-a}{\norm{b-a}}
\end{gather}
and set $\al_1:=\ang{v_1,u}$ and $\al_2:=-\ang{v_2,u}$.
If $\al_1\ge0$, then $d(u,\R_+v_1)=\sqrt{1-\al_1^2}$; otherwise $d(u,\R_+v_1)=1$.
Hence, in both cases it holds $d^2(u,\R_+v_1)+(\al_1)_+^2=1$.
Similarly, $d^2(-u,\R_+v_2)+(\al_2)_+^2=1$.
Equality (i) follows from the definitions.

(ii) 
Let $\itr<\ga<\frac{1}{\sqrt{2}}$ and choose a $\ga'>0$ and an $\eps>0$ such that $\itr<\ga'<\ga$ and
\begin{gather}\label{P19P-8}
2\ga'\sqrt{1-(\ga')^2}+\eps<2\ga\sqrt{1-\ga^2},
\end{gather}
which is possible because the function $\ga'\mapsto\ga'\sqrt{1-(\ga')^2}$ is increasing on $[0,\frac{1}{\sqrt{2}}]$.
By the second representation in \eqref{std+2}, there exist points $a\in(A\setminus B)\cap\B_\eps(\bx)$, $b\in(B\setminus A)\cap\B_\eps(\bx)$, $x\in\B_\eps(\bx)$, $v_1\in N_{A}^p(a)$, $v_2\in N_{B}^p(b)$ such that $a\ne x$, $b\ne x$, $1-\eps<\frac{\norm{x-a}}{\norm{x-b}}<1+\eps$, $\norm{v_1}=\norm{v_2}=1$, $\norm{u_1-v_1}<\eps$, $\norm{u_2-v_2}<\eps$, and $\norm{u_1+u_2}<2\ga'$, where
\begin{gather}\label{P19P-5}
u_1:=\frac{x-a}{\norm{x-a}},\quad
u_2:=\frac{x-b}{\norm{x-b}}.
\end{gather}
Employing the notations \eqref{P19P-3} and \eqref{P19P-5}, set
\begin{gather}\label{P19P-1}
\al:=\ang{u_1,u_2},\quad \al_1:=\ang{u_1,u},\quad \al_2:=-\ang{u_2,u},
\\\label{P19P-2}
\be:=\norm{u_1+u_2},\quad
\be_1:=\sqrt{1-\al_1^2},\quad \be_2:=\sqrt{1-\al_2^2}.
\end{gather}
The relationship between the numbers $\al$ and $\be$ is straightforward:
\begin{gather}\label{P19P-4}
\be^2=2(1+\al).
\end{gather}
Let $\mathfrak{A}$, $\mathfrak{A}_1$ and $\mathfrak{A}_2$ stand for the angles between $u_1$ and $u_2$, $u$ and $u_1$, and $u_2$ and $-u$, respectively (measured counterclockwise).
Then $\mathfrak{A}+\mathfrak{A}_1+\mathfrak{A}_2=\pi$, $\sin\mathfrak{A}_1\ge0$, $\sin\mathfrak{A}_2\ge0$, and
\begin{gather*}
\cos\mathfrak{A}=-\cos(\mathfrak{A}_1+\mathfrak{A}_2) =-\cos\mathfrak{A}_1\cos\mathfrak{A}_2+
\sin\mathfrak{A}_1\sin\mathfrak{A}_2.
\end{gather*}
Hence,
\begin{gather}\label{P19P-7}
\al=\cos\mathfrak{A}\ge -\cos\mathfrak{A}_1\cos\mathfrak{A}_2 =-\al_1\al_2.
\end{gather}
By assumption, $\be<2\ga<\sqrt{2}$.
It follows from \eqref{P19P-4} that $\al<0$ (i.e. angle $\mathfrak{A}$ is obtuse), and consequently,
$\al_1>0$ and $\al_2>0$.
Since $\al_1\le1$ and $\al_2\le1$ (see \eqref{P19P-1}), we have $\al_1\al_2\le\hat\al:=\min\{\al_1,\al_2\}$, and in view of \eqref{P19P-7}, $-\hat\al\le\al$.
Using \eqref{P19P-4} again, we have
\begin{gather}\label{P19P-6}
0\le\mu:=1-\hat\al\le\frac{\be^2}{2} <2(\ga')^2<1.
\end{gather}
In view of \eqref{P19P-2} and \eqref{P19P-6} and taking into account that the function $\mu\mapsto\sqrt{\mu(2-\mu)}$ is increasing on $[0,1]$, it holds
\begin{gather}\label{P19P-9}
\hat\be:=\max\{\be_1,\be_2\} =\sqrt{1-\hat\al^2} =\sqrt{\mu(2-\mu)}\le2\ga'\sqrt{1-(\ga')^2}.
\end{gather}
Set $\hat v_1:=\al_1v_1$, $\hat v_2:=\al_2v_2$, $\hat u_1:=\al_1u_1$ and $\hat u_2:=\al_2u_2$, and notice that $\hat v_1\in N_{A}^p(a)$, $\hat v_2\in N_{B}^p(b)$, $\norm{\hat u_1-u}=\sqrt{1-\al_1^2}$ and $\norm{\hat u_2+u}=\sqrt{1-\al_2^2}$.
Hence, in view of \eqref{P19P-9} and \eqref{P19P-8},
\begin{gather*}
\max\{\norm{\hat v_1-u},\norm{\hat v_2+u}\}< \hat\be+\eps \le2\ga'\sqrt{1-(\ga')^2}+\eps <2\ga\sqrt{1-\ga^2}.
\end{gather*}
It follows from the definitions~\eqref{std4} and \eqref{P19P-3} that $\itrdh{1}\le2\ga\sqrt{1-\ga^2}$.
Letting $\ga\downarrow\itr$, we arrive at the claimed inequality.

(iii) If $\itr=0$, then $\itrdh{1}=0$ in view of part (ii).
Let $\itrdh{1}=0$.
By the definition \eqref{std4}, for any $\eps>0$,
there exist points $a\in(A\setminus B)\cap\B_\eps(\bx)$, $b\in(B\setminus A)\cap\B_\eps(\bx)$, $v_1\in N_{A}^p(a)$ and $v_2\in N_{B}^p(b)$ such that $\norm{u-v_1}<\eps$ and $\norm{u+v_2}<\eps$ where $u$ is given by \eqref{P19P-3}.
Without loss of generality, we can assume that $\norm{v_1}=\norm{v_2}=1$.
Set $x:=(a+b)/2$.
Then $x\ne a$, $x\ne b$, $\norm{x-a}=\norm{x-b}$, and employing the notations \eqref{P19P-5}, $u_1=u$ and $u_2=-u$.
Hence, $\norm{u_1-v_1}<\eps$ and $\norm{u_2-v_2}<\eps$, and it follows from the first representation in \eqref{std+2} that $\itr=0$.
\qed\end{proof}

\begin{remark}
For the expression in the \RHS\ of the inequality in Proposition~\ref{P19}(ii), we have the following estimates:
\begin{gather*}
0\le2\itr\sqrt{1-(\itr)^2}\le1
\end{gather*}
as long as $0\le\itr\le1$.
It equals 0 if and only if either $\itr=0$ or $\itr=1$.
It equals 1 if and only if $\itr=\frac{1}{\sqrt{2}}$.
\xqed\end{remark}

Thanks to Proposition~\ref{P19}, the criteria of intrinsic transversality in Theorems~\ref{T3} and \ref{T5} can be complemented by several more characterisations collected in the next theorem.

\begin{theorem}\label{T6}
Suppose $X$ is a Euclidean space,
$A,B\subset X$ are closed, and $\bx\in A\cap B$.
The following conditions are equivalent:
\begin{enumerate}
\item
$\{A,B\}$ is intrinsically transversal at $\bar x$;
\item
$\itrdh{1}>0$,\\
i.e., there exist numbers $\alpha\in]0,1[$ and $\de>0$ such that
\begin{gather*}
\max\biggl\{d\left(\frac{b-a}{\|a-b\|}, N_{A}(a)\right),
d\left(\frac{a-b}{\|a-b\|}, N_{B}(b)\right)\biggr\}>\alpha
\end{gather*}
for all $a\in(A\setminus B)\cap\B_\de(\bx)$, $b\in(B\setminus A)\cap\B_\de(\bx)$, or equivalently,\\
there exist numbers $\alpha\in]0,1[$ and $\de>0$ such that
\begin{gather*}
\max\biggl\{\norm{\frac{b-a}{\|a-b\|}-v_1},
\norm{\frac{a-b}{\|a-b\|}-v_2}\biggr\}>\alpha
\end{gather*}
for all $a\in(A\setminus B)\cap\B_\de(\bx)$, $b\in(B\setminus A)\cap\B_\de(\bx)$, and all $v_{1}\in N_{A}(a)$ and $v_{2}\in N_{B}(b)$;

the exact upper bound of all such $\al$ equals $\itrdh{1}$.
\item
$\itrdh{2}<1$,\\
i.e., there exist numbers $\alpha\in]0,1[$ and $\de>0$ such that
\begin{gather*}
\min\biggl\{\ang{\frac{b-a}{\|a-b\|},v_1},
\ang{\frac{a-b}{\|a-b\|},v_2}\biggr\}<\alpha
\end{gather*}
for all $a\in(A\setminus B)\cap\B_\de(\bx)$, $b\in(B\setminus A)\cap\B_\de(\bx)$, and all $v_{1}\in N_{A}(a)$ and $v_{2}\in N_{B}(b)$ with $\norm{v_1}=\norm{v_2}=1$;

the exact lower bound of all such $\al$ equals $\itrdh{2}$.
\end{enumerate}
\end{theorem}

\begin{remark}
1. Conditions $N_{A}(a)\ne\{0\}$ and $N_{B}(b)\ne\{0\}$ can be added in Theorem~\ref{T6}(i) (cf. Remark~\ref{R14}).

2. Thanks to Theorem~\ref{T6} and taking into account Remark~\ref{R14}.2, Definition~\ref{D3+}(ii) of intrinsic transversality formulated in a general normed linear space, in the Euclidean space setting reduces to the original definition of this property introduced recently by Drusvyatskiy et al. \cite{DruIofLew15} (see \cite[Definition~3.1 and formula~(3.1)]{DruIofLew15} and \cite[formula~(5)]{NolRon16}).

3. The six constants providing quantitative characterisations of the intrinsic transversality property of $\{A,B\}$ at $\bx$ make two distinct groups: 1) $\itr$, $\itrd{1}$, $\itrd{2}$, $\itrd{3}$ and 2)~$\itrdh{1}$, $\itrdh{2}$.
Within each group, the constants can be easily converted from one into another thanks to Proposition~\ref{P18} and Proposition~\ref{P19}(i).
The constants belonging to different groups are not convertible.
We only have a one-sided estimate in Proposition~\ref{P19}(ii) complemented by the fact in Proposition~\ref{P19}(iii) that constants $\itr$ and $\itrdh{1}$ can equal zero only simultaneously.
Fortunately the last fact is sufficient for detecting intrinsic transversality qualitatively.

4. Compared to $\itrdh{1}$, the definition \eqref{itr} of $\itr$ contains an additional parameter: $x$ which in a sense determines the ``directions'' of the normal vectors $x_1^*$ and $x_2^*$.
This explains why the constants belonging to different groups are not convertible (see the previous remark) and seems to be an advantage of the definition \eqref{itr} when characterizing the intrinsic transversality property quantitatively as it eliminates normal vectors which are irrelevant from the point of view of intrinsic transversality.
\xqed\end{remark}

The next theorem provides a list of equivalent criteria of subtransversality of a pair of convex sets which follow from Theorems~\ref{T5} and \ref{T6} in view of Corollary~\ref{C4+}.

\begin{theorem}
Suppose $A,B\subset X$ are closed and convex, and $\bx\in A\cap B$. Then
$\{A,B\}$ is subtransversal at $\bar x$ if and only if one of the conditions {\rm (ii)--(v)} in Theorem~\ref{T5} or
{\rm (ii)--(iii)} in Theorem~\ref{T6} is satisfied.
\end{theorem}

\section{Conclusions and future work}\label{S5}

A connection has been established between the two seemingly different normal cone transversality properties of pairs of nonconvex sets: the one introduced in Kruger et al \cite{KruLukTha} as a sufficient condition of subtransversality in Asplund spaces and the finite dimensional Euclidean space \emph{intrinsic transversality} property introduced in Drusvyatskiy et al. \cite{DruIofLew15} as a sufficient condition for local linear convergence of alternating projections for solving feasibility problems.
It is shown that in Euclidean spaces the properties are equivalent.
Several characterizations of this property are established.
Two new limiting objects are used in the finite dimensional characterizations: the cone of pairs of relative limiting normals and the cone of pairs of restricted relative limiting normals.
They possess certain similarity with the conventional limiting normal cones, but unlike the latter one are defined for pairs of sets.
Special attention is given to the convex case.

The following questions need to be answered and have been identified for future research.
The readers are welcome to contribute.
\renewcommand{\labelenumi}{\rm \arabic{enumi})}
\begin{enumerate}
\item
Does the dual characterization in Theorem~\ref{T1'} reduce to that in Theorem~\ref{T2} when the sets are convex?
\item
Can Proposition~\ref{P10}(v) be extended to general Banach spaces?
\item
The relationship between intrinsic transversality and weak intrinsic transversality should be further investigated.
Are they different in general? in finite dimensions? in Euclidean spaces?
\item
When do the sets defined in the two parts of Definition~\ref{D5} coincide?
\item
An analogue of Theorem~\ref{T3} for weak intrinsic transversality should be formulated.
\item
It is not important for estimating intrinsic transversality, but it would be good to add the case $\itr\ge\frac{1}{\sqrt{2}}$ to Proposition~\ref{P19} for completeness.
\end{enumerate}

\begin{acknowledgements}
The author thanks Nguyen Hieu Thao for many constructive comments and suggestions regarding several definitions and statements in the article, and the referees for the careful reading of the manuscript and constructive comments and suggestions.
\end{acknowledgements}
\def\cprime{$'$} \def\cftil#1{\ifmmode\setbox7\hbox{$\accent"5E#1$}\else
  \setbox7\hbox{\accent"5E#1}\penalty 10000\relax\fi\raise 1\ht7
  \hbox{\lower1.15ex\hbox to 1\wd7{\hss\accent"7E\hss}}\penalty 10000
  \hskip-1\wd7\penalty 10000\box7} \def\cprime{$'$} \def\cprime{$'$}
  \def\cprime{$'$} \def\cprime{$'$} \def\cprime{$'$}
  \def\Dbar{\leavevmode\lower.6ex\hbox to 0pt{\hskip-.23ex \accent"16\hss}D}
  \def\cfac#1{\ifmmode\setbox7\hbox{$\accent"5E#1$}\else
  \setbox7\hbox{\accent"5E#1}\penalty 10000\relax\fi\raise 1\ht7
  \hbox{\lower1.15ex\hbox to 1\wd7{\hss\accent"13\hss}}\penalty 10000
  \hskip-1\wd7\penalty 10000\box7} \def\cprime{$'$}


\end{document}